\documentclass[11pt,a4paper]{article}

\usepackage[utf8]{inputenc}
\usepackage[T1]{fontenc}
\usepackage{lmodern}

\usepackage{amsmath,amsfonts,amssymb,amsthm}
\usepackage{mathtools}
\usepackage{accents}

\usepackage{hyperref}
\hypersetup{
    colorlinks=true,
    linkcolor=blue,
    citecolor=red,
    urlcolor=blue,
    pdfborder={0 0 0}
}

\usepackage[a4paper]{geometry}

\usepackage{cleveref}
  \crefname{section}{Section}{Sections}
  \crefname{figure}{Figure}{Figures}
  \crefname{theorem}{Theorem}{Theorems}
  \crefname{lemma}{Lemma}{Lemmas}
  \crefname{proposition}{Proposition}{Propositions}  
  \crefname{corollary}{Corollary}{Corollaries}
  \crefname{definition}{Definition}{Definitions}
  \crefname{example}{Example}{Examples}
  \crefname{remark}{Remark}{Remarks}

\newtheorem{theorem}{Theorem}[section]
\newtheorem{lemma}[theorem]{Lemma}
\newtheorem{proposition}[theorem]{Proposition}
\newtheorem{corollary}[theorem]{Corollary}

\newtheorem{definition}[theorem]{Definition}

\newtheorem{remark}[theorem]{Remark}

\usepackage{enumerate} 
\usepackage{enumitem}

\newtheorem{condition}{Condition}

\numberwithin{equation}{section}

\DeclarePairedDelimiter\abs{\lvert}{\rvert}
\DeclarePairedDelimiter\norm{\lVert}{\rVert}

\newcommand*{\defeq}{\mathrel{\mathop:}=}

\newcommand*{\ubar}[1]{\underaccent{\bar}{#1}}

\newcommand*{\dist}{\operatorname{dist}}
\newcommand*{\crad}{\operatorname{crad}}

\newcommand{\R}{\ensuremath{\mathbf{R}}}

\newcommand{\E}{\ensuremath{\mathbf{E}}}
\renewcommand{\P}{\ensuremath{\mathbf{P}}}

\newcommand*{\pr}{\mathbf{P}}
\newcommand*{\ex}{\mathbf{E}}

\newcommand*{\NN}{\mathbf{N}}

\newcommand*{\RR}{\mathbf{R}}

\newcommand*{\HH}{\mathbf{H}}

\title{On Loewner chains driven by semimartingales and
complex Bessel-type SDEs}

\author{
Vlad Margarint
\thanks{UNC Charlotte.
    Email: \url{vmargari@charlotte.edu}} 
\and
Atul Shekhar%
  \thanks{TIFR-CAM, Bangalore.
    Email: \url{atul@tifrbng.res.in}}
\and
Yizheng Yuan
\thanks{U Cambridge and TU Berlin.
Email: \url{yy547@cam.ac.uk}}
}

\begin{document}

\maketitle

\begin{abstract}
We prove existence (and simpleness) of the trace for both forward and backward Loewner chains under fairly general conditions on semimartingale drivers. As an application, we show that stochastic Komatu-Loewner evolutions SKLE$_{\alpha,b}$ are generated by curves. As another application, motivated by a question of A. Sep\'{u}lveda, we show that for $\alpha >3/2$ and Brownian motion $B$, the driving function $|B_t|^\alpha$ generates a simple curve for small $t$. On a related note we also introduce a complex variant of Bessel-type SDEs and prove existence and uniqueness of strong solution. Such SDEs appear naturally while describing the trace of Loewner chains. In particular, we write SLE$_\kappa$, $\kappa <4$, in terms of stochastic flow of such SDEs.  
\end{abstract}

\section{Introduction} \label{intro}
The Loewner theory is a key ingredient in the construction of Schramm-Loewner evolutions (SLEs), see \cite{law-conformal-book} and references therein for a detailed introduction to SLEs and its numerous applications in statistical mechanics. A Loewner chain is a family $\{g_t\}_{t \in [0,T]}$ of conformal maps $g_t\colon \mathbf{H}\setminus K_t \to \mathbf{H}$, where $\mathbf{H}$ is the upper half plane and the family $\left\{K_{t}\right\}_{t \in[0, T]}$ are subsets of $\mathbf{H}$ satisfying a certain local growth property, see e.g. \cite[Section 7]{bn-sle-lecture} for a precise definition. Such families $\left\{K_{t}\right\}_{t \in[0, T]}$ are in a one-to-one correspondence with real-valued continuous functions $\left\{W_{t}\right\}_{t \in [0, T]}$ which we refer to as its \textit{driving function} or simply its \textit{driver}. When $W_t$ is chosen to be $\sqrt{\kappa}B_t$, where $B_t$ is a standard Brownian motion, it gives rise to the SLE$_{\kappa}$ curves. More precisely, it was proven in \cite{rs-sle,lsw-lerw-ust} that the Loewner chain driven by $\sqrt{\kappa} B_{t}$ is generated by a continuous curve $\gamma^{\kappa}\colon [0,T] \to \overline{\mathbf{H}}$, i.e. for $t >0$, $\mathbf{H} \setminus K_t$ is the unbounded component of $\mathbf{H}\setminus \gamma[0,t]$. The curve $\gamma^{\kappa}$ is defined as SLE$_{\kappa}$. It is a simple curve if and only if $\kappa \leq 4.$ In this article we ask ourselves what happens if we replace Brownian motion by continuous semimartingales (such a study was initiated in \cite{fs-le-rp} and it is further extended here). More precisely, we ask the following questions:

\begin{enumerate}
\item Are Loewner chains driven by semimartingales generated by a curve?

\item When are these curves simple?
\end{enumerate}

As a motivation for considering semimartingale drivers and asking above questions, we show in Section \ref{app-sec} that stochastic Komatu-Loewner evolutions (which are variants of SLEs in finitely connected domains) are generated by curves. We also illustrate in Section \ref{app-sec} other prospective applications.

\smallskip

We remind ourselves that there are two possible ways of generating sets $K_t$ using the driver $W_t$: either by using the forward Loewner differential equation (LDE) which describes the evolution of conformal maps $g_t\colon \mathbf{H}\setminus K_t \to \mathbf{H}$ or by using the backward LDE which describes the evolution of $g_{T-t}\circ g_T^{-1}\colon \mathbf{H} \to \mathbf{H}\setminus g_{T-t}(K_T)$. The forward LDE is driven by $W_t$, and the backward LDE is driven by the time reversal $U$ of $W$, i.e. $U_t = W_T - W_{T-t}$. In the case of SLE where the driver is a Brownian motion, since Brownian motion is time reversible, addressing above questions using the forward or the backward LDE are equivalent. However, if we want to replace Brownian motion by continuous semimartingales, since time reversal of a semimartingale need not be a semimartingale, we have to distinguish between the forward and the backward case. We therefore consider two scenarios: (i) $W_t$ is a semimartingale, and, (ii) when $U_t$ is a semimartingale. We ask in each scenarios questions (a)-(b). The case of the forward semimartingale corresponds to growing a random curve from inside that changes its $\kappa$ parameter according to its past. The backward LDE, on the other hand, can be interpreted as a conformal welding process that changes its $\kappa$ parameter according to the previous welding (cf. \cite{rz-backward-sle,she-quantum-zipper}). The main result of this article gives answers to the above questions under fairly general conditions on the semimartingale driver.

More precisely, we consider semimartingales satisfying the following conditions. Let $T>0$ and $(\Omega, \{\mathcal{F}_t\}_{t\in [0,T]}, \mathcal{F}, \P)$ be a filtered probability space satisfying the usual hypothesis and let $S_t$ be a continuous semimartingale defined on it. We write $S_t = M_t + A_t $, where $M$ is a local martingale and $A$ is a bounded variation process. For the sake of simplicity we also assume that the filtration $\{\mathcal{F}_t\}_{t\in[0,T]}$ is rich enough to support a Brownian motion on it (results of this paper are valid even without this assumption and we have assumed this just to avoid some cumbersome notations). It follows using the martingale representation theorem that $M_t  = \int_0^t \sqrt{\kappa_s}dB_s$ for some Brownian motion $B$ and an adapted process $\kappa_s$ (we are therefore defining SLEs with non-constant $\kappa$ which can itself be random). Our assumption is as follows. 

\begin{condition}\label{sem-cond}
Let $S_t = M_t+A_t$ be a semimartingale as above where $M_t = \int_0^t \sqrt{\kappa_s}\,dB_s$. Suppose
\begin{enumerate}[label= (\roman*)]

\item $\kappa_s \in [\ubar\kappa, \bar\kappa]$ for some constants $\ubar\kappa,\bar\kappa$ that are either $0= \ubar\kappa\leq \bar\kappa <8$ or $8 < \ubar\kappa \leq \bar\kappa < \infty$.

\item $A \in W^{1,2}[0,T]$, i.e.\ $\dot{A}_t = dA_t/dt \in L^2[0,T]$ almost surely.

\end{enumerate}

\end{condition}

Our main results are the following theorems.

\begin{theorem}\label{main-thm-1}
If $U$ is a semimartingale satisfying Condition \ref{sem-cond}, then the Loewner chain with the driver $W$ given by $W_t = U_T - U_{T-t}$, $t \in [0,T]$, is almost surely generated by a curve $\gamma$. Furthermore, if $\bar\kappa <4$, then $\gamma$ is almost surely simple and $\gamma_t \in \mathbf{H}$ for $t \in {(0,T]}$.
\end{theorem}

\begin{theorem}\label{main-thm-2}
If $W$ is a semimartingale satisfying Condition \ref{sem-cond}, then the Loewner chain with the driver $W$ is almost surely generated by a curve $\gamma$. Furthermore, if either $\bar\kappa \leq 4$ with $A\equiv 0$ or $\bar\kappa <4$ with possibly non-zero $A$, then $\gamma$ is almost surely simple and $\gamma_t \in \mathbf{H}$, $\forall t>0$. 
\end{theorem}

\begin{corollary}\label{imp-cor}
If either $U_t = |B_t|^\alpha$ or $W_t = |B_t|^\alpha$ for some $\alpha > 3/2$, then the Loewner chain driven by $W$ is a.s.\ generated by a curve $\gamma$ for $t < \inf\bigl\{ s \ge 0 \bigm| |B_s|^{\alpha-1} \geq \alpha^{-1}\sqrt{8} \bigr\}$. Moreover, $\gamma$ is a simple curve for $t < \inf\bigl\{ s \ge 0 \bigm| |B_s|^{\alpha-1} \geq \alpha^{-1}\sqrt{4} \bigr\}$.
\end{corollary}

\smallskip

Some remarks are in order. 

\begin{itemize}

\item The existence of $\gamma$ in Theorem \ref{main-thm-1} for the special case $\bar\kappa <2$ was proven in \cite{fs-le-rp}. 

\item The distinguishing feature of our proofs as compared to proofs in \cite{rs-sle} is that Theorem \ref{main-thm-1} is solely based on backward flow analysis and Theorem \ref{main-thm-2} is solely based on forward flow analysis.\footnote{The proof of existence of $\gamma$ in \cite{rs-sle} is based on the backward flow analysis and the proof of simpleness of $\gamma$ is based on the forward flow analysis.} While the existence part in Theorem \ref{main-thm-1} is an adaptation of the argument in \cite{rs-sle} (together with an additional observation, see Section \ref{sketch-proof} below; we also prove an uniform estimate for $\bar\kappa <4$, see \eqref{keybound}, \eqref{theta exists} below), the simpleness part, being based on the backward flow analysis, is new to this article and substantially different (and long to our own surprise) from \cite{rs-sle}. On the other hand, the simpleness part in Theorem \ref{main-thm-2} is an adaptation of the argument in \cite{rs-sle}, but existence part, being based on the forward flow analysis, is new to this article. This proof is interesting in its own right and it gives a new proof of the existence of $\gamma$ for SLE$_\kappa$. Furthermore, this idea can be further extended to obtain refined (variation and Hölder-type) regularity estimates for SLE$_\kappa$ that include and add logarithmic refinements to the results in \cite{vl-sle-hoelder,ft-sle-regularity}, see \cite{yuan-psivarx} for details. Another application in the context of L\'{e}vy process driven Loewner chains is given in \cite{ps-le-levyx}.

\item The Condition \ref{sem-cond} on $A$ might suggest a role of Girsanov Theorem in our proofs, but this is rather not the case since we do not assume a lower bound away from zero on $\kappa_s$. The finite energy drivers are special in this context because of totally different reasons, see \cite{wang-energy-reversibility,wang-energy-equivalent,vw-welding-flow} for a detailed study of such drivers and their special properties. 
\item Theorem \ref{main-thm-1} and Theorem \ref{main-thm-2} are different from each other in general situations. But, one can deduce one result from the other when the semimartingale is reversible. This itself consists a large class of stochastic processes e.g. many diffusion processes, see \cite[Chapter 6]{pro-stochastic-book}, \cite{my-enlargement-book} for the related \textit{expansion of filtration} technique. 

\end{itemize}

\medskip

Techniques developed for the proof of Theorem \ref{main-thm-1} and Theorem \ref{main-thm-2} also allows us to prove the following additional result which provides a conceptual clarification in the story of existence of $\gamma$. It was proven in \cite{rs-sle} that the curve $\gamma$ exists if and only if $\lim _{y \rightarrow 0+} f_{t}\left(i y+W_{t}\right)$ exists and the limit is continuous in $t$ (and we will in fact use this criteria to prove the above results). In such cases the curve $\gamma$ is given by $\gamma_t = \lim _{y \rightarrow 0+} f_{t}\left(i y+W_{t}\right)$. It is then natural to ask if it is possible to identify the limit $\lim _{y \rightarrow 0+} f_{t}\left(i y+W_{t}\right)$ in terms of a canonical intrinsically defined object. We achieve this goal in terms of certain Bessel-type SDEs as follows. It is well known that that a description of $f_t(iy+W_t)$ can be obtained by solving the backward LDE started from $iy$, see \eqref{f=h} below. It is therefore natural to consider backward LDE started from zero. But the backward LDE started from zero is a singular equation, and it is a priori not well-defined. To get around this issue we follow the approach of \cite[Chapter 11]{ry-stochastic-book} which deals with a similar situation while making sense of Bessel processes starting from zero. More precisely, let $V_t$ be a semimartingale satisfying Condition \ref{sem-cond} and $h_t(iy)$ be the solution to the backward LDE
\begin{align}\label{Eq1}
d h_{t}=d V_{t}-\frac{2}{h_{t}} d t, \hspace{2mm} h_0=iy.
\end{align}

When $h_0 =0$, the idea is to consider the squared equation and then define $h_t(0)$ by taking square root. But since we expect these solutions to be complex-valued, we have to work with the complex square root function which can be multivalued. This prompts us to make the following definition. Let  
$$\sqrt{z}=\operatorname{sgn}(\operatorname{Im}(z)) \sqrt{\frac{|z|+\operatorname{Re}(z)}{2}}+i \sqrt{\frac{|z|-\operatorname{Re}(z)}{2}}$$ which is a bijection from $\mathbf{C}\setminus [0, \infty)$ to $\mathbf{H}$. Let $\varphi\colon[0,T] \to \mathbf{C}$ be a continuous function.

\begin{definition} For a continuous function (respectively continuous adapted process) $\varphi\colon[0, T] \rightarrow \mathbf{C}$, a branch square-root of $\varphi$ is a continuous (respectively continuous adapted) function $\theta\colon[0, T] \rightarrow \overline{\mathbf{H}}$ such that $\theta_{t}^{2}=\varphi_{t}$,
$\forall t \in[0, T]$. We then write $\theta_{t}=\sqrt{\varphi_t}^{b}.$
\end{definition}

Note that the only ambiguity while choosing a branch square root is when $\varphi_t \in (0,\infty)$. In that case a branch square root makes a choice from $\pm \sqrt{|\varphi_t|}$ in a continuous adapted way.

We consider the It\^o SDE 
\begin{equation}\label{Eq2}
d\varphi_t=2\sqrt{\varphi_t}^{b}dV_t+ (\kappa_t-4)dt, \hspace{2mm} \varphi_0=0,
\end{equation}
where $\sqrt{\varphi}^{b}$ is a branch of square root of $\varphi$ (compare this to the Bessel SDE where it is $\sqrt{|\varphi_t|}$ and therefore the solution is forced to be real valued). Our following theorem establishes the existence and uniqueness of strong solutions to such SDEs when $\bar\kappa <4$. \footnote{Note that Yamada-Watanabe Theorem does not apply below because complex square root $\sqrt{z}$ is not a $1/2$-H\"older function.} 

\begin{theorem}\label{main-thm-3}
If $V_t = \int_0^t \sqrt{\kappa_s}\,dB_s+A_t$ is a semimartingale satisfying condition \ref{sem-cond} with $\bar\kappa <4$, then
\begin{itemize}
\item[a)] If $\varphi$ is a solution to \eqref{Eq2}, then a.s. $\forall t>0$, $\varphi_t \in \mathbf{C} \setminus [0, \infty)$. In particular, $\sqrt{\varphi_t}^b=\sqrt{\varphi_t}$, and \eqref{Eq2} is equivalent to 
\begin{equation}\label{Eq3}
d\varphi_t=2\sqrt{\varphi_t}dV_t+ (\kappa_t-4)dt, \hspace{2mm} \varphi_0 =0.
\end{equation}
\item[b)] There exists a continuous adapted process satisfying \eqref{Eq3}. Moreover, if $\varphi$ and $\tilde{\varphi}$ are two such solutions, then 
\[\mathbf{P}\left[\varphi_{t}=\tilde{\varphi}_{t}, \hspace{2mm} \forall t \geqslant 0\right]=1.\]
\end{itemize} 
\end{theorem}

We define the solution to $h_t(0)$ of \eqref{Eq1} started from zero as $h_t(0)=\sqrt{\varphi_t}$. \footnote{Compare Theorem \ref{main-thm-3} with a result of Krylov-R\"ockner \cite {kr-singular-drift} which considers multi-dimensional SDEs with singular drifts. Equation \eqref{Eq1} started from zero can be viewed as a $2$-dimensioal SDE with a singular drift. A distinction between Theorem \ref{main-thm-3} and the result of \cite{kr-singular-drift} is that the noise term in Theorem \ref{main-thm-3} is only one-dimensional.}

\begin{remark}
Note that it follows that $\forall t\geq t_0>0$
$$h_{t}(0)=h_{t_{0}}(0)+V_{t}-V_{t_{0}}-\int_{t_{0}}^t \frac{2}{h_r(0)} d r.$$ 
It implies in particular that  
\begin{equation}\label{**}
\int_{0+}^{t} \frac{2}{h_{r}} d r :=\lim_{t_{0} \rightarrow 0+} \int_{t_{0}}^{t} \frac{2}{h_{r}} d r,
\end{equation}
exists. Hence, we obtain  
\begin{equation}\label{Eq-3}
h_t=V_{t}-V_0-\int_{0+}^{t} \frac{2}{h_{r}} d r.
\end{equation}
However, we do not know whether 
$\int_{0+}^{t} \frac{2}{|h_{r}|} d r < \infty$, and this makes \eqref{Eq-3} inconvenient to deal with directly. In particular, we do not know if the solution $h_t(0)$ is a semimartingale (compare it with real Bessel processes with dimension less than $1$ which are not semimartingales). 
\end{remark}

In the setup of Theorem \ref{main-thm-1}, let $U$ be the time reversal of the driving function $W$.
To represent the curve $\gamma$ in terms of the above SDE, we consider the stochastic flow of equation \eqref{Eq1} in $\overline{\mathbf{H}}$ by choosing $V_t = U_t - U_s$, i.e. for $0\leq s\leq t \leq T$ and $z \in \overline{\mathbf{H}}$, let $h(s,t,z)$ denote the solution of 
\begin{equation}\label{floweqn}
h(s, t, z)=z+U_{t}-U_{s}-\int_{s}^{t} \frac{2}{h(s, r, z)} d r, \hspace{2mm} h(s, s, z)=z.
\end{equation}
When $z\in \mathbf{H}$, the solution $h$ is classically well defined for all $t\geq s$. For $z=0$, we define $h(s,t,0) = \sqrt{\varphi_t}$, where $\varphi_t$ is as constructed by Theorem \ref{main-thm-3}. For $z\in \R \setminus \{0\}$, the solution is classically well-defined until the time the solution hits zero. We then again continue the solution for further time according to Theorem \ref{main-thm-3}. We will prove below that $h(s,t,0^{\downarrow}) := \lim_{y\to 0+} h(s,t,iy)$ exists a.s. uniformly in $s,t$, see \eqref{theta exists} below. It then follows easily that $\varphi_t := \lim_{y\to 0+} h(s,t,iy)^2$ is a solution to \eqref{Eq3}, and uniqueness of solution in Theorem \ref{main-thm-3} implies that for any $s,t$ we have $h(s,t,0^{\downarrow}) = h(s,t,0)$ almost surely. In particular, $h(\cdot,\cdot,0^{\downarrow})$ is a continuous modification of the random field $h(\cdot,\cdot,0)$.

Using \eqref{f=h}, we obtain: 

\begin{corollary}[SLEs as stochastic flows]\label{corr-1}
The curve $\gamma$ in Theorem \ref{main-thm-1} for $\bar\kappa <4$ is given by
\[ \gamma_{t}=h(T-t, T, 0^{\downarrow}) \]
which is a continous version of the process $t \mapsto h(T-t, T, 0)$.

In particular, for $\kappa <4$, the law of the SLE$_{\kappa}$ restricted to $[0,T]$ is the same as that of $h^{\kappa}(T-t, T, 0)$, where $h^{\kappa}$ is the stochastic flow driven by $\sqrt{\kappa}B.$

\end{corollary}

Some remarks are again in order:

\begin{itemize}

\item In connection to SLE$_\kappa(\rho)$ processes with applications in Liouville Quantum Gravity similar half-plane valued solutions to Bessel type SDEs have also been considered in \cite[Proposition 3.8]{dms-mating-treesx} where the existence and uniqueness of weak solutions was established. The Theorem \ref{main-thm-3} establishes the strong uniqueness. The SDE \eqref{Eq2} is not expected to have unique solution if the condition $\bar\kappa <4$ is not satisfied. For example, if $\kappa_s = \kappa =4$, a trivial solution is $\varphi \equiv 0$. One can however construct non-zero solutions by examining SLE$_4$. This is very similar to a related work of Bass-Burdzy-Chen \cite{BBC07} where they establish the uniqueness of solution to certain degenerate SDEs under the assumption that the solution spends zero time at zero. Similarly, when $\kappa_s = \kappa >4$, we can construct real solutions by examining usual Bessel processes, and complex solutions by examining SLE$_\kappa$, $\kappa >4$.
  
\item The Corollary \ref{corr-1} has an interesting implication. Note that the classical construction of SLE$_\kappa$ identifies these curves as boundary of simply connected domains which are obtained by evolving $\HH$ under the flow of backward LDE. The important implication of Corollary \ref{corr-1} is that in the particular case of $\kappa<4$, one can evolve only the boundary $\partial \HH = \R$ under the backward flow according to Theorem \ref{main-thm-3} and recover the SLE$_\kappa$ curve. We therefore do not have to evolve the interior of $\overline{\HH}$ and there is no ``loss of information" while passing to the boundary. In particular, this implies that SLE$_\kappa$ for $\kappa <4$ is a measurable w.r.t. $\{h^\kappa(0,T,x)\}_{x\in \R}$. We believe that this is very closely related to the fact that SLE$_\kappa$ is conformally removable for $\kappa <4$, see \cite{rs-sle}, \cite{js-removability} for details. This simplification of only evolving the boundary $\R$ is not enough to recover SLE$_\kappa$ for $\kappa>4$. In this case one genuinely has to evolve $\HH$ and recover SLE$_\kappa$ as boundary the evolved domain. This makes the problem of identifying the limit $\lim_{y\to 0+} f_t(iy + W_t)$ more complicated when $\kappa >4$. We plan to address this in our future projects.

\item Theorem \ref{main-thm-3} is a basis for the follow up work \cite{LDP-Bessel} which proves the large deviation principle for solutions $\varphi$ and $\kappa \to 0$. This has been motivated by previous works of Y. Wang \cite{wang-energy-reversibility,wang-energy-equivalent}.

\end{itemize}

\subsection{Sketch of ideas behind the proofs} \label{sketch-proof}

The proof of existence in both Theorem \ref{main-thm-1} and Theorem \ref{main-thm-2} will follow the outline of \cite{rs-sle}. In particular, we will verify in both cases that 
\begin{equation}\label{RS-C}
\gamma_t := \lim _{y \rightarrow 0+} f_{t}\left(i y+W_{t}\right)
\end{equation} 
exists and it is continuous in $t$. Also, in order to verify this, we will rely on the derivative bound 
\begin{equation}\label{req-bound}
\sup_t |f_t'(iy + W_t)| \lesssim y^{-\theta}
\end{equation} 
for some $\theta <1$. By \cite[Corollary 3.12]{vl-sle-hoelder}, this condition is sufficient for the Loewner chain to be generated by a continuous curve.

In Theorem \ref{main-thm-1}, the bound \eqref{req-bound} is obtained similarly as in \cite{rs-sle} by obtaining certain moment estimates on $|f_t'(iy + W_t)|$. In \cite{rs-sle}, they represent $f_t = g_t^{-1}$ using the backward flow $h_t$ (see \cref{prem} for details). Due to \eqref{formula}, the expectations $\ex\abs{h_t'(z)}^\lambda$ can be computed by solving a Feynman-Kac formula.
It turns out that it is easier to compute weighted moments of the form $\E[|h_t'(z)|^\lambda Y_t^{-\lambda}F(X_t/Y_t, Y_t)]$ for some appropriate function $F$, where $h_t(z)= X_t + iY_t$. It is also convenient to work with the coordinates $(w,y) = (x/y, y)$. In the case when $U= \sqrt{\kappa}B$, one can look for functions $F$ so that one can explicitly compute $ \E[|h_t'(z)|^\lambda Y_t^{-\lambda}F(X_t/Y_t, Y_t)]$. For $F = F(w,y) \in C^2$ we see from It\^o formula that
\begin{multline}\label{nc:eq:bw_M_ito}
d\left(|h_t'(z)|^\lambda Y_t^{-\lambda} F(\frac{X_t}{Y_t},Y_t)\right) 
= |h_t'(z)|^\lambda Y_t^{-\lambda-2} \\
\left[ \left( -\frac{4\lambda}{(1+X_t^2/Y_t^2)^2}F +\frac{2Y_t}{1+X_t^2/Y_t^2}F_y -\frac{4X_t/Y_t}{1+X_t^2/Y_t^2}F_w +\frac{\kappa_t}{2}F_{ww} \right) dt \right.\\
\left.\vphantom{\frac{X_t^2}{Y_t^2}} - \sqrt{\kappa_t}Y_t F_w \, dB_t \right]
\end{multline}

Finding an appropriate $F$ can be boiled down to solving a PDE $\Lambda^{\mathrm{(bw)}}_\kappa F =0$, where 
\begin{equation}\label{eq:diff_op_bw}
\Lambda^{\mathrm{(bw)}}_\kappa F := -\frac{4\lambda}{(1+w^2)^2}F +\frac{2y}{1+w^2}F_y -\frac{4w}{1+w^2}F_w +\frac{\kappa}{2}F_{ww} . 
\end{equation}    

An explicit solution to the above PDE is given by $F(w,y) = (1+w^2)^{r/2} y^{\zeta+\lambda}$, where $r, \zeta, \lambda$ are related by 
\begin{equation}\label{r-def}
\lambda = r(1+\frac{\kappa}{4})-\frac{r^2 \kappa}{8},\quad \zeta = r-\frac{r^2 \kappa}{8}.
\end{equation}
When $\kappa$ is not constant, or if there is a drift term in the semimartingale $U$, the above explicit computation is not feasible. In such cases, the problem of bounding $ \E[|f_t'(iy+ W_t)|^\lambda Y_t^{-\lambda}F(X_t/Y_t, Y_t)]$ for $\kappa_t \in [\ubar\kappa,\bar\kappa]$ can be interpreted as an optimal stochastic control problem. We would need to solve a Hamilton-Jacobi-Bellman-type equation
\[ \sup_{\kappa \in [\ubar\kappa,\bar\kappa]} \Lambda_\kappa F = 0 . \]

Usually one cannot hope for an explicit solution to such equations. But it turns out that one can find supersolutions
\begin{equation}\label{eq:bw_supersol}
    \sup_{\kappa \in [\ubar\kappa,\bar\kappa]} \Lambda_\kappa F \le 0.
\end{equation}
In fact one can construct supersolutions by changing appropriately parameters $r, \lambda, \zeta$ in $F(w,y) = (1+w^2)^{r/2} y^{\zeta+\lambda}$, see Section \ref{proof-thm-1-exists} for details 
(Cf. \cite{bs-aims-sle,cr-le-symmetric-stable,ps-le-levyx} for similar ideas in slightly different settings. We are also reminded of computing superhedging prices under uncertain volatility, cf. \cite{jm-uncertain-vol}). The basic observation of this paper is that it is enough to construct supersolutions for estimating moments of $|h_t'(z)|$. 

\smallskip

For Theorem \ref{main-thm-2} we would again like to find the bound \eqref{req-bound} for $|f_t'(iy + W_t)|$. But since we have no good way of working with the time reversal of $W$, we do not have the backward Loewner flow at our disposal. We therefore have to work with $g_t'(z)$ instead of $f_t'(z)$. Obtaining (negative) moment estimates on $|g_t'(z)|$ is similar to the above computations. But, to obtain \eqref{req-bound} from the moment estimates on $|g_t'(z)|$ requires an additional idea as follows. Let $\delta>0$. We want to find an upper bound for $|f_t'(i\delta + W_t)| = |g_t'(f_t(i\delta + W_t))|^{-1}$. Observe that $z = f_t(i\delta + W_t)$ is the point where we have to start the forward flow $\{Z_s\}_{s\in [0,t]}$ in order to reach $Z_t = i\delta$. Of course the flow $\{Z_s\}_{s\in [0,t]}$ depends on the behaviour of $W$ in the time interval $[0,t]$. That means we would need to consider all possible points $z \in \mathbf{H}$ that might reach $i\delta$ at time $t$. It turns out that, using Koebe's distortion estimates, we can reduce the number of points and we only need to start the flow from a finite number of points. The number of points so needed will encode the information on $|f_t'(i\delta + W_t)|$, see Section \ref{proof-thm-2-exists} for the implementation of this idea. 

It is worth mentioning that this approach of analysing the trace directly via the forward flow also applies to usual SLE$_\kappa$ (with constant $\kappa$). In fact, it is used in \cite{yuan-psivarx} to obtain refined regularity statements for usual SLE$_\kappa$.

\smallskip

The simpleness part in Theorem \ref{main-thm-2} is done similarly as in \cite{rs-sle}. It boils down to proving that forward flow started at $x>0$ stays positive for all time. We prove this by simply adapting the corresponding proof for Bessel processes. On the other hand, the simpleness part in Theorem \ref{main-thm-1} requires a more careful analysis. It boils down to prove that if $T(s, x)=\inf\{t \geqslant s \mid h(s, t, x)=0\},$ then $T(s,x)$ is almost surely jointly continuous in $(s,x) \in [0,T]\times [0,\infty)$. This requires us to a priori verify that $T(s,x) <\infty$ a.s., and we in fact prove that it has finite moment of order $p>1$. The joint continuity of $T(s,x)$ is established via a covering argument: We cover $[0,T]$ using intervals of form $[s_n,s_{n+1}]$, where $s_{n+1} =T(s_n,x_n)$. It turns out that the number of intervals needed to cover $[0,T]$ is of order $x_n^{-2}$. Since we have $p>1$ moments, we can also ensure that each $s_{n+1} - s_{n}$ is small. This implies that for all $s\in [s_n, s_{n+1}]$ and $x$ small enough, $T(s,x)$ is also small, see Section \ref{proof-thm-1-simple} for details. 

\smallskip

Throughout the rest of the paper, we are going to assume that our semimartingales satisfy (in addition to Condition~\ref{sem-cond})
\begin{equation}\label{energy}
\mathbf{E}\left[\exp \left\{\sigma \int_{0}^{T} \dot{A}_{r}^{2} d r\right\}\right]<\infty
\end{equation}
where $\sigma$ is a sufficiently large constant (depending on $\ubar\kappa, \bar\kappa$). This incurs no loss of generality since our theorems are almost sure statements, hence by defining $\widehat{S}_t \defeq M_t+A_{t \wedge \tau_n}$ where $\tau_n = \inf\{ t \mid \int_0^t \dot{A}_{r}^{2} d r = n \}$ we see that $S = \widehat{S}$ on the event $\{ \int_{0}^{T} \dot{A}_{r}^{2} d r \le n \}$, and letting $n \to \infty$ yields the general statements.

\subsection{Organization of the paper}\label{organisation}

We recall some preliminary facts on Loewner chains in Section \ref{prem}. Theorem \ref{main-thm-1} is proved in Section \ref{proof-thm-1}; the existence and simpleness are divided into subsection \ref{proof-thm-1-exists} and subsection \ref{proof-thm-1-simple}. Theorem \ref{main-thm-2} is proved in Section \ref{proof-thm-2}, with subsections \ref{proof-thm-2-exists} and \ref{proof-thm-2-simple} devoted to existence and simpleness respectively. The Corollary \ref{imp-cor} is proven in Section \ref{proof-imp-cor}. Finally, the Theorem \ref{main-thm-3} is proved in Section \ref{proof-thm-3}. In Section \ref{app-sec} we provide an application of Theorem \ref{main-thm-2} to SKLEs, and discuss some other prospective applications. 

\subsection{Acknowledgements}

The authors would like to thank Peter Friz and Steffen Rohde for many discussions and comments. We thank the referees for their careful reading and their comments. 
V.M. acknowledges the support of NYU-ECNU Institute of Mathematical Sciences at NYU Shanghai. A.S. acknowledges the support from the project PIC RTI4001: Mathematics, Theoretical Sciences and Science Education. Y.Y. acknowledges the support from European Research Council through Starting Grant 804166 (SPRS; PI: Jason Miller), and partial support through Consolidator Grant 683164 (PI: Peter Friz) during the initial stage of this work at TU Berlin.

\section{Preliminaries} \label{prem}

We will often write $a \lesssim b$ meaning $a \le Cb$ for some constant $C < \infty$ that may depend on the context and change from line to line. We write $a \asymp b$ when $a \lesssim b$ and $b \lesssim a$. 

We recall some basic facts about Loewner chains. More details can be found e.g. in \cite{law-conformal-book,kem-sle-book}  
We will consider Loewner chains $\{g_t\}_{t\in [0,T]}$ parametrised by its half plane capacity, i.e. the mapping out functions $g_t\colon\mathbf{H} \setminus K_t \to \mathbf{H}$ satisfies  
$$g_{t}(z)=z+\frac{2 t}{z}+O\left(\frac{1}{|z|^{2}}\right) \hspace{2mm} \text{as} \hspace{1mm} z \to \infty.$$
The driving function of $\{ g_t \}$ is given by 
$$W_{t}:=\cap_{h > 0} \overline{g_{t}\left(K_{t+h} \backslash K_{t}\right)}.$$

The maps $g_t$ satisfy the (forward) Loewner Differential Equation 
\begin{equation}\label{LDE}
\partial_{t} g_{t}(z)=\frac{2}{g_{t}(z)-W_{t}}, \quad g_{0}(z)=z.
\end{equation}

It will be convenient for us to work with $Z_t(z) = g_t(z)- W_t$. Writing $Z_t(z) = R_t + i I_t$, it follows that
\begin{equation} \begin{split}\label{RI-def}
    dR_t &= -dW_t + \frac{2R_t}{R_t^2+I_t^2}dt,\\
    dI_t &= \frac{-2I_t}{R_t^2+I_t^2}dt,
\end{split} \end{equation}
and
\begin{equation*}
 |g_t'(z)| = \exp \left( -2 \int_0^t \frac{R_s^2-I_s^2}{(R_s^2+I_s^2)^2} \, ds \right).\\
\end{equation*}

The inverse map $f_t(z)=g_t^{-1}(z)\colon \mathbf{H} \to \mathbf{H}\setminus K_t$ can be obtained by solving the following backward LDE. Let $h(s,t,z)$ denote the \textit{flow of solutions} of equation 
\begin{equation}\label{b-LDE-flow}
 h(s, t, z)=z+U_{t}-U_{s}-\int_{s}^{t} \frac{2}{h(s, r, z)} d r,
 \end{equation}

where $U$ is the time reversal of $W$, i.e. $U_t = W_T - W_{T-t}$.
Note that $h(\cdot,\cdot,\cdot)$ satisfies the so called \textit{flow property}, i.e. for all $s\leq u \leq t$,
\[ h(s,t, z) = h(u,t, h(s,u,z)).\]
\vspace{2mm}
It can be then easily checked that 
\begin{equation}\label{f=h}
f_{t}(z+ W_t)=h(T-t, T, z),
\end{equation}
see e.g. \cite[Lemma 2.1]{stw-finite-var} for details. We will also sometimes write $\hat{f}_t(z)$ for $f_t(z + W_t)$.

\smallskip

The following exact formulas will be useful for us to check the condition \eqref{req-bound}. Writing $z = x+ iy$ and $h(s,t, z) = h_t(z)= X_t+iY_t$ (we will often drop the index $s, z$ to avoid the cumbersome notation, and its dependence on $s,z$ will be understood from the context), the equation \eqref{b-LDE-flow} is equivalently written as 
\begin{equation}\label{X-eqn}
dX_{t}=dU_t-\frac{2X_t}{X_t^{2}+Y_t^{2}} d t, \hspace{2mm} X_s = x, 
\end{equation}
\begin{equation}\label{Y-eqn}
dY_{t}=\frac{2 Y_{t}}{X_{t}^{2}+Y_{t}^{2}} d t, \hspace{2mm} Y_s = y. 
\end{equation}
 
 Also, it can be easily checked that  
\begin{equation}\label{formula}
\left|h^{\prime}(s, t, z)\right|=\exp \left\{\int_{s}^{t} \frac{2\left(X_{r}^{2}-Y_{r}^{2}\right)}{\left(X_{r}^{2}+Y_{r}^{2}\right)^{2}} d r\right\}. 
\end{equation}

Another key tool that will be important for our proof is the Gr\"onwall inequality. We will use in the following slightly unconventional form.
\begin{lemma}[Gr\"onwall inequality]\label{Gronwall}
Let $F(t,x)$ be a bounded continuous function that is continuously differentiable in $x$ with $\partial_x F(t,x) \geq 0$. Let $L_t$ be a continuous function such that 
$$
L_{t} \leq \int_{0}^{t} F\left(r, L_{r}\right) d r,$$ and $R_t$ be a continuous function satisfying 
\begin{equation}
R_{t}=\int_{0}^{t} F\left(r, R_{r}\right) d r.
\end{equation}
Then, $$L_t \leq R_t, \hspace{2mm} \forall t \geq 0.$$

\end{lemma}

\section{Proof of Theorem \ref{main-thm-1}} \label{proof-thm-1}

As we explained at the end of Section~\ref{sketch-proof}, we are going to assume that $U$ is a semimartingale satisfying Condition~\ref{sem-cond} and that additionally \eqref{energy} holds.

\subsection{Proof of existence of $\gamma$} \label{proof-thm-1-exists}
In the following we fix $T>0$. We will obtain an estimate of the form 
\begin{equation}\label{prime-estimate}
\sup_{0\leq s \leq T}|h'(s,T,iy)| \lesssim y^{-\theta}  
\end{equation}
for some $\theta < 1$. Then, using \eqref{f=h}, the bound \eqref{req-bound} follows, which is well known to imply the existence of $\gamma$ (cf. \cite[Corollary 3.12]{vl-sle-hoelder}). 
We also include Proposition \ref{crucial-estimate}-$(b)$ which will be important in the proof of simpleness of $\gamma$.

\begin{proposition}\label{crucial-estimate}
Suppose that $U$ satisfies Condition~\ref{sem-cond} and \eqref{energy}.
\begin{enumerate}
\item For $s \ge 0$, let
\[ M_t = |h_t'(z)|^\lambda Y_t^\zeta (1+X_t^2/Y_t^2)^{r/2}, \quad t \ge s. \]
There exist $r, \lambda, \zeta$ such that 
\begin{equation}\label{parameter-cond}
r >0,\quad \lambda >0,\quad \lambda + \zeta >2,
\end{equation} 
and 
\begin{equation}\label{M-bound}
\sup_{s,t,z} \E[M_t/M_0] < \infty. 
\end{equation}

\item If $\bar\kappa < 4$, there exists $\lambda>2$ such that  
\begin{equation}\label{keybound}
\sup _{s, y} \mathbf{E}\left[\sup _{t}\left|h^{\prime}(s, t, i y)\right|^{\lambda}\right]<\infty.
\end{equation}
\end{enumerate}
\end{proposition}

\begin{proof}[Proof of Proposition \ref{crucial-estimate}-$(a)$]

If the drift part $A\equiv 0$, we can follow exactly the strategy described in \cref{sketch-proof}. Let $F(w,y) = (1+w^2)^{r/2} y^{\zeta+\lambda}$ as above. Then $M_t =|h_t'(z)|^\lambda Y_t^{-\lambda} F(X_t/Y_t,Y_t)$. Recalling \eqref{eq:diff_op_bw}, a calculation reveals that $\Lambda^{\mathrm{(bw)}}_\kappa F \le 0$ on $\HH$ if and only if
\begin{equation}\label{eq:bw_supersol_cond}
\lambda-\zeta \ge \frac{r\kappa}{4}
\quad\text{and}\quad
\lambda+\zeta \le 2r+\frac{r\kappa}{4}-\frac{r^2 \kappa}{4}.
\end{equation}
If this is satisfied for all $\kappa \in [\ubar\kappa, \bar\kappa]$, then by \eqref{nc:eq:bw_M_ito}, we have that $(M_t)$ is a non-negative local supermartingale, and therefore a supermartingale. Therefore, $\E[M_t] \leq M_0$ which gives the required bound. We now show how to pick $r, \lambda, \zeta$ satisfying \eqref{eq:bw_supersol_cond}.

In the case $\bar\kappa < 8$, we pick $\lambda$ and $\zeta$ according to \eqref{r-def}, and $r = \frac{1}{2}+\frac{4}{\bar\kappa}$, in which case $\zeta+\lambda > 2$.

In the case $\ubar\kappa > 8$, we will pick $r \in [0,1]$, in which case the condition \eqref{eq:bw_supersol_cond} follows from $\lambda-\zeta \ge \frac{r\bar\kappa}{4}$ and $\lambda+\zeta \le 2r+\frac{r\ubar\kappa}{4}-\frac{r^2 \ubar\kappa}{4}$. Picking $\zeta+\lambda > 2$ is now possible if and only if $2r+\frac{r\ubar\kappa}{4}-\frac{r^2 \ubar\kappa}{4} > 2 \iff r \in {]\frac{8}{\ubar\kappa},1[}$. With any such $r$, we can then pick $\lambda = r+\frac{r(\bar\kappa+\ubar\kappa)}{8}-\frac{r^2 \ubar\kappa}{8}$ and $\zeta = r-\frac{r(\bar\kappa-\ubar\kappa)}{8}-\frac{r^2 \ubar\kappa}{8}$ which satisfy \eqref{eq:bw_supersol_cond}.

\smallskip

To handle the drift term, we will need an additional argument as follows. Applying It\^o formula to $\log(X_t^2  + Y_t^2)$, we obtain that 
\begin{align*}
\log \left(X_{t}^{2}+Y_{t}^{2}\right)
={}& \log(x^2 + y^{2})+ \int_{s}^{t} \frac{2 X_{u}}{X_{u}^{2}+Y_{u}^{2}} d U_{u} -\int_{s}^{t} \frac{2 X_{u}^2}{(X_{u}^{2}+Y_{u}^{2})^2} d[U]_ u  \nonumber\\
&-\int_{s}^{t} \frac{4(X_u^2-Y^2_{u})}{(X_{u}^{2}+Y_{u}^{2})^2}du +\int_{s}^{t} \frac{1}{X_{u}^{2}+Y_{u}^{2}} d[U]_{u}.
\end{align*}

We also note using \eqref{Y-eqn} that 
\[ Y_t = y \exp\biggl[ \int_s^t \frac{2}{X_u^2 + Y_u^2}du\biggr].\]
It therefore follows using \eqref{formula} that 
\begin{equation} \label{h'=exp}
|h_t'(z)|^\lambda Y_t^{\zeta} (1+X_t^2/Y_t^2)^{r/2} =  |h_t'(z)|^\lambda Y_t^{\zeta-r} (X_t^2 + Y_t^2)^{r/2} = y^{\zeta - r}(x^2+ y^2)^{r/2} \exp\bigl[\Theta_t \bigr],
\end{equation}
where 
\[
\Theta_t := \int_s^t \frac{rX_u}{X_u^2 + Y_u^2}dU_u - \int_s^t \frac{\kappa_u r^2 X_u^2}{2(X_u^2 + Y_u^2)^2}du + \int_s^t \frac{\alpha_uX_u^2}{(X_u^2 + Y_u^2)^2}du  + \int_s^t \frac{\beta_uY_u^2}{(X_u^2 + Y_u^2)^2}du,
\]

\vspace{2mm}
with \[\alpha_u := 2\lambda + 2 \zeta - 4r - \frac{\kappa_u r}{2} + \frac{\kappa_u r^2}{2},\] 
and 
\[\beta_u:= \frac{\kappa_u r}{2} + 2 \zeta - 2 \lambda.\]

Again, if the drift part $A\equiv 0$, the condition \eqref{eq:bw_supersol_cond} implies $\alpha_u \le 0$ and $\beta_u \le 0$. Then
\[\Theta_t \leq \int_s^t \frac{rX_u}{X_u^2 + Y_u^2}dU_u - \int_s^t \frac{\kappa_u r^2 X_u^2}{2(X_u^2 + Y_u^2)^2}du, \]
and 
\begin{align}\label{h-exp}
|h_t'(z)|^\lambda & Y_t^{\zeta} (1+X_t^2/Y_t^2)^{r/2} \\ \nonumber &\leq y^{\zeta - r}(x^2+ y^2)^{r/2} \exp\biggl[ \int_s^t \frac{rX_u}{X_u^2 + Y_u^2}dU_u - \int_s^t \frac{\kappa_u r^2 X_u^2}{2(X_u^2 + Y_u^2)^2}du\biggr].
\end{align}
Note that in case $U$ is a martingale, we see again that the right hand side in the above equation a positive local martingale, hence a supermartingale, and we reprove the above obtained bound $\E[M_t] \leq M_0$.

To handle the drift part $A$, observe that we can vary the parameters a bit so that \eqref{eq:bw_supersol_cond} is satisfied with strict inequalities. Then $\alpha_u \leq \alpha < 0$ and $\beta_u \leq \beta <0$ for some negative constants $\alpha, \beta$. This allows us to estimate
\[
\int_s^t \frac{rX_u}{X_u^2 + Y_u^2}dA_u = \int_s^t \frac{rX_u}{X_u^2 + Y_u^2}\dot{A}_udu \leq \frac{\delta^{2}r^2}{2} \int_{s}^{t} \frac{X_u^{2}}{\left(X_u^{2}+Y_{u}^{2}\right)^{2}} d u+\frac{1}{2\delta^{2}} \int_{s}^t \dot{A}_{u}^{2} du .
\]
If $\delta > 0$ is small enough, the first term can be absorbed into $\alpha$ where the second term has exponential moments by our condition \eqref{energy}.

We then have
\[
\Theta_t \le \int_s^t \frac{rX_u}{X_u^2 + Y_u^2}\sqrt{\kappa_u}dB_u + \int_s^t \left(-\frac{\kappa_u r^2}{2}+\alpha+\frac{\delta^{2}r^2}{2}\right) \frac{X_u^2}{(X_u^2 + Y_u^2)^2}du + \frac{1}{2\delta^{2}} \int_{s}^t \dot{A}_{u}^{2} du .
\]
Now, pick $p,q \in (1, \infty)$ with $p^{-1} + q^{-1} =1$, and apply H\"older's inequality to obtain that 
\begin{multline*}
\E[\exp(\Theta_t)] \leq \\
\E\biggl[\exp\biggl(\int_s^t \frac{prX_u}{X_u^2 + Y_u^2}\sqrt{\kappa_u}dB_u + p\int_s^t \left(-\frac{\kappa_u r^2}{2}+\alpha+\frac{\delta^{2}r^2}{2}\right) \frac{X_u^2}{(X_u^2 + Y_u^2)^2}du\biggr)\biggr]^{1/p} \times \\ 
\E\biggl[\exp\biggl( \frac{q}{2\delta^{2}} \int_{s}^t \dot{A}_{u}^{2} du\biggr)\biggr]^{1/q}.
\end{multline*}
Furthermore, writing the first term on the right-hand side as
\begin{multline*}
\E\biggl[\exp\biggl(\int_s^t \frac{prX_u}{X_u^2 + Y_u^2}\sqrt{\kappa_u}dB_u - \int_s^t \frac{\kappa_u p^2 r^2 X_u^2}{2(X_u^2 + Y_u^2)^2}du \\
+ \int_s^t \left(p\alpha + \frac{p\delta^{2}r^2}{2} + (p^2-p)\frac{\kappa_u r^2}{2}\right)\frac{X_u^2}{(X_u^2 + Y_u^2)^2}du\biggr)\biggr],
\end{multline*}
we again note similarly as above that 
\[\exp\biggl(\int_s^t \frac{prX_u}{X_u^2 + Y_u^2}\sqrt{\kappa_u}dB_u - \int_s^t \frac{\kappa_u p^2 r^2 X_u^2}{2(X_u^2 + Y_u^2)^2}du\biggl)\] 
is a supermartingale. Picking $\delta > 0$ small and $p > 1$ sufficiently close to $1$ (recall $\alpha < 0$), we finally obtain
\[ \E[\exp(\Theta_t)] \leq \E\biggl[\exp\biggl(\frac{q}{2\delta^{2}} \int_{s}^t \dot{A}_{u}^{2} du\biggr)\biggr]^{1/q},\]
which gives us the required bound \eqref{M-bound}.
\end{proof}

\begin{proof}[Proof of Proposition \ref{crucial-estimate}-$(b)$]
If we also insist on $\zeta >0$, then \eqref{h-exp} implies 
\begin{align}
|h_t'(iy)|^\lambda &\leq (1+X_t^2/Y_t^2)^{-r/2} \bigl(\frac{y}{Y_t}\bigr)^{\zeta}\exp\biggl[ \int_s^t \frac{rX_u}{X_u^2 + Y_u^2}dU_u - \int_s^t \frac{\kappa_u r^2 X_u^2}{2(X_u^2 + Y_u^2)^2}du\biggr] \\
& \leq \exp\biggl[ \int_s^t \frac{rX_u}{X_u^2 + Y_u^2}dU_u - \int_s^t \frac{\kappa_u r^2 X_u^2}{2(X_u^2 + Y_u^2)^2}du\biggr]. 
\end{align}
In case $A=0$,
by Dambis-Dubins-Schwarz (DDS) martingale embedding theorem, we note that \[\sup_t \biggl[ \int_s^t \frac{rX_u}{X_u^2 + Y_u^2}dU_u - \int_s^t \frac{\kappa_u r^2 X_u^2}{2(X_u^2 + Y_u^2)^2}du\biggr]\] 
is stochastically dominated by $\sup_{t\geq 0} (B_t - \frac{t}{2})$. It is well known that $\sup_{t\geq 0} (B_t - \frac{t}{2})$ is an exponential random variable with parameter $1$, see e.g. \cite[Exercise 3.12]{ry-stochastic-book}.
It particular, it has finite exponential moments of order $p<1$. We then simply note that $\bar\kappa <4$ allows us to have $\lambda >2$ in the above estimates. 
The case of non-zero $A$ is handled similarly as above.
\end{proof}

\begin{remark}\label{sup-remark}
The proof above can be seen as an instance of the It\^o-Tanaka trick. Note that $|h_t'(z)|$ is a bounded variation process in $t$. We however express it in terms of a stochastic integral using the It\^o Lemma. This allows us to use techniques from stochastic calculus to estimate $|h_t'(z)|$. We can compare the above proof with the method of Krylov-R\"ockner \cite{kr-singular-drift} which establishes strong uniqueness of solutions to various singular SDEs. The same argument as above also gives that when $A=0$ and $\kappa$ is constant (even if $\kappa =8$),
\[\sup _{s, y} \mathbf{P}\left[\sup _{t}\left|h^{\prime}(s, t, i y)\right| \geq K\right] \leq \frac{C}{K^2}.\]
 But the above quadratic tail estimate is not enough to obtain \eqref{theta exists}.
\end{remark}

It is well known in the literature that the estimate \eqref{M-bound} implies the estimate \eqref{prime-estimate}. The following corollary is essentially the same as \cite[Corollary 3.5]{rs-sle}. For the convenience of the reader, we repeat it here with the slight adaptions to our case. We also include the bound \eqref{theta exists} which follows similarly from \eqref{keybound}. This will be important in the proof of simpleness of $\gamma$. 

\begin{corollary}\label{nc:thm:bw_hprime_probability}
Suppose that $U$ satisfies Condition~\ref{sem-cond} and \eqref{energy}.
\begin{enumerate}
\item Suppose $r,\lambda, \zeta$ are chosen according to Proposition \ref{crucial-estimate}-$(a)$. Then
\begin{equation*} 
\mathbf{P}(|h_t'(iy)| \ge u) \lesssim 
\begin{cases}
 u^{-(\zeta+\lambda)} & \text{if } \zeta > 0,\\
 u^{-\lambda} y^\zeta & \text{if } \zeta < 0.
\end{cases}
\end{equation*}
for all $0 \le s \le t$ and $y \in {]0,1]}$.\\
In particular the bound \eqref{prime-estimate} holds. 
\item For $\bar\kappa <4$, let $\lambda >2$ be chosen from Proposition \ref{crucial-estimate}-$(b)$, then 
\begin{equation}\label{apply-markov}
\mathbf{P}(\sup_t |h_t'(iy)| \ge u) \lesssim u^{-\lambda}.
\end{equation} 

In particular, there exists $\theta <1$ such that 
\begin{equation}\label{theta exists}
\sup _{0 \leq s \leq t \leq T}\left|h^{\prime}(s, t, i y)\right| \lesssim y^{-\theta}.
\end{equation}
 
\end{enumerate}
\end{corollary}

\begin{proof}

For the part $(a)$, recall that the equation \eqref{Y-eqn} implies $Y_t \le \sqrt{y^2+4t}$. Moreover, by the Schwarz lemma we have $|h_t'(iy)| \le \frac{Y_t}{y}$, therefore $|h_t'(iy)| \ge u$ implies $Y_t \ge yu$. Hence,
\begin{align*}
\mathbf{P}(|h_t'(iy)| \ge u) 
&\le \sum_{m=\lfloor\log u \rfloor}^{\lceil\frac{1}{2}\log(1+4t/y^2)\rceil} \mathbf{P}\left(|h_t'(iy)|\ge u,\ Y_t \in [ye^{m-1},ye^m]\right)\\
&\lesssim \sum_{m=\lfloor\log u \rfloor}^{\lceil\frac{1}{2}\log(1+4t/y^2)\rceil} u^{-\lambda} y^{-\zeta}e^{-m\zeta} \E[M_t] \\
&\lesssim \sum_{m=\lfloor\log u \rfloor}^{\lceil\frac{1}{2}\log(1+4t/y^2)\rceil} u^{-\lambda} e^{-m\zeta} \\
&\lesssim u^{-\lambda} \begin{cases}
u^{-\zeta} & \text{if } \zeta > 0,\\
(1+4t/y^2)^{-\zeta/2} & \text{if } \zeta < 0
\end{cases}\\
&\lesssim \begin{cases}
u^{-\zeta-\lambda} & \text{if } \zeta > 0,\\
u^{-\lambda} y^\zeta & \text{if } \zeta < 0.
\end{cases}
\end{align*}
The bound \eqref{prime-estimate} follows using a Borel-Cantelli argument. More precisely, if $\zeta >0$ (resp. $\zeta <0$), we choose $\theta<1$ such that $(\lambda + \zeta)\theta >2$ (resp. $\lambda\theta + \zeta >2$). It then follows from the above that  
\begin{equation}\label{finite-sum}
 \sum_{n} \sum_{k=0}^{2^{2n}} \mathbf{P}( |h'(k2^{-2n}, T, i2^{-n})| \geq 2^{n\theta}) < \infty.
 \end{equation}
The Borel-Cantelli lemma implies that almost surely 
\begin{equation}\label{discrete}
 |h'(k2^{-2n}, T, i2^{-n})| \leq 2^{n\theta}
 \end{equation} for all $n$ large enough. Using the fact that $U$ is weakly $1/2$-H\"older \footnote{It can be easily checked e.g. using the DDS martingale embedding theorem that semimartigales satisfying Condition \ref{sem-cond} are indeed weakly $1/2$-H\"older.}, the bound \eqref{prime-estimate} follows from \eqref{discrete} using the classical distortion theorem, see \cite[Section 3]{vl-sle-hoelder} for details. 

\smallskip

For the part $(b)$, the bound \eqref{apply-markov} follows easily from \eqref{keybound} and the Markov inequality. The uniform estimate \eqref{theta exists} follows from the distortion theorem similarly as above. 
\end{proof}

\begin{remark}
In the case $\ubar\kappa > 8$, one can ask whether there are smarter ways of finding supersolutions to \eqref{eq:bw_supersol} that are sharper. Looking at the proofs of \cite{vl-sle-hoelder,ft-sle-regularity}, the optimal regularity of the SLE that can be proved are directly related to the exponents $r,\lambda,\zeta$ (there is a bit more freedom for $r$ though). It is reasonable to believe that the regularity of $\gamma$ in our case should be the same as for SLE$_{\ubar\kappa}$. One possible attempt to prove such a thing would be to find a supersolution to $\sup_{\kappa \in [\ubar\kappa,\bar\kappa]} \Lambda_\kappa F \le 0$ that is asymptotically comparable to $(1+w^2)^{r/2} y^{\zeta+\lambda}$ at least for $y \searrow 0$ (where $\lambda,\zeta$ are chosen accordingly with $\kappa = \ubar\kappa$).

Under certain conditions on $\ubar\kappa$, $\bar\kappa-\ubar\kappa$, and $r$, a function of the form
\[ F(w,y) = y^{\zeta+\lambda}(1+w^2)^{r/2} \exp(g(w)) \]
with some bounded $g$ indeed does the trick. More precisely, we can pick $g$ such that $g'$ is of the form
\[ g'(w) = 
\begin{cases}
-\alpha_1 w & \text{for } w \le w_0,\\
-\alpha_2 w^{-1-\varepsilon} & \text{for } w \ge w_0.
\end{cases}
\]
This works whenever $\bar\kappa-\ubar\kappa$ is sufficiently small (depending on $\ubar\kappa,r$). Unfortunately, we did not succeed in making this work in general.
\end{remark}

\subsection{Proof of simpleness of $\gamma$}  \label{proof-thm-1-simple}
We now prove the simpleness of the curve $\gamma$ when $\bar\kappa <4$. We first need to prove some preparatory lemmas. It will be convenient for us to extend the definition of $U_t$ for all $t \in [0, \infty)$ by defining $U_t=U_T$, $\forall t \geq T$. This naturally extends the definition of $h(s,t,z)$ for all $0\leq s\leq t < \infty$. We will also need to consider the stochastic flow $h(s, t, x)$ started at $x>0$ defined by \eqref{b-LDE-flow}. Note that the solution $h(s,t,x)$ is only well-defined for $t <T(s,x)$, where 
\begin{equation}
T(s, x)=\inf\{t \geqslant s \mid h(s, t, x)=0\}. 
\end{equation}
It is a priori not clear that $T(s,x) <\infty$ a.s., but we prove that this is indeed the case. The following proposition is the most important step in the proof of simpleness of $\gamma$. 

\begin{proposition}\label{Proposition10}
If $U$ is a semimartingale satisfying Condition \ref{sem-cond} with $\bar\kappa <4$ and \eqref{energy}, then there exists $p>1$ depending only on $\kappa, \sigma$ such that 
\begin{equation}
\sup _{s, x} \mathbf{E}\left[\left|\frac{T(s, x)-s}{x^{2}}\right|^{p}\right]<\infty.
\end{equation}
\end{proposition}

\begin{proof}
Applying It\^o formula to $\log(h_t(x))$, it follows that
\begin{align*}
\log h_{t}(x)&=\log (x)+\int_{s}^{t} \frac{1}{h_{r}(x)} d U_{r}-\int_{s}^{t} \frac{2}{h_{r}(x)^{2}} d r-\frac{1}{2} \int_{s}^{t} \frac{1}{h_{r}(x)^{2}} d[U]_{r}\\
&=\log (x)+\int_{s}^{t} \frac{\sqrt{\kappa_r}}{h_{r}(x)} dB_r+\int_{s}^{t} \frac{1}{h_{r}(x)} \dot{A}_{r} d r-
\int_{s}^{t} \frac{2}{h_{r}(x)^{2}} d r-\frac{1}{2} \int_{s}^{t} \frac{\kappa_r}{h_{r}(x)^{2}} dr\\
&\leq 
\log (x)+\int_{s}^{t} \frac{\sqrt{\kappa_r}}{h_r(x)} dB_{r}-\frac{1}{2} \int_{s}^{t} \frac{\kappa_r}{h_r(x)^{2}} dr\\
&\quad -\left(2-\frac{\delta^{2}}{2}\right) \int_{1}^{t} \frac{1}{h_{r}(x)^{2}} d r+\frac{1}{2 \delta^{2}} \int_{s}^{t} \dot{A}_{r}^{2} d r,
\end{align*}
where in the last line we have used 
\begin{align*}
\frac{\dot{A}_{r}}{h_{r}(x)} \leqslant \frac{\delta^{2}}{2} \frac{1}{h_{r}(x)^{2}}+\frac{1}{2 \delta^{2}} \dot{A}_{r}^{2},
\end{align*}
for some $\delta$ small enough. Next, for $\eta>0$ small enough, we write
\begin{equation}
\left(2-\frac{\delta^{2}}{2}\right) \int_{s}^{t} \frac{1}{h_r(x)^{2}} d r=\left(2-\frac{\delta^{2}}{2}-\eta\right) \int_{s}^{t} \frac{1}{h_{r}(x)^{2}} d r+\eta\int_{s}^t \frac{1}{h_{r}(x)^{2}} d r.
\end{equation}
Using $\kappa_r \leq \bar\kappa$, note that 
\begin{equation}
\left(2-\frac{\delta^{2}}{2}-\eta\right) \int_{s}^{t} \frac{1}{h_{r}(x)^{2}} d r \geqslant \frac{1}{\bar\kappa}\left(2-\frac{\delta^{2}}{2}-\eta\right) \int_{s}^{t} \frac{\kappa_r }{h_{r}(x)^{2}} dr.
\end{equation}
Therefore, we obtain that for all $t <T(s,x)$,
\begin{align}\label{logh-bound}
\log h_t(x) 
&\leqslant \log (x)+\int_{s}^{t} \frac{\sqrt{\kappa_r}}{h_{r}(x)} dB_{r}-\left\{\frac{1}{2}+\frac{1}{\bar\kappa}\left(2-\frac{\delta^{2}}{2}-\eta\right)\right\} \int_{s}^{t} \frac{\kappa_r}{h_r(x)^{2}} dr\nonumber\\
&\quad +\frac{1}{2 \delta^{2}} \int_{s}^{t} \dot{A}_{r}^{2} d r-\int_{s}^{t} \frac{\eta}{h_{r}(x)^{2}} d r.
\end{align}
Now, let \begin{align*}
T_{\varepsilon}(s, x)=\inf \{t \geqslant s \mid h(s, t, x)=\varepsilon\}.
\end{align*}
Then, for $t \leq T_{\varepsilon}(s, x)$ it follows from \eqref{logh-bound} that 
\begin{equation}
\log h_{t}(x) \leqslant \log (x)+\Theta-\int_{s}^{t} \frac{\eta}{h_{r}(x)^{2}} d r,
\end{equation}
where $\Theta= \Theta_1 + \Theta_2$ with 
\[
\Theta_1 := \operatorname{sup}_{t \leq T_{\varepsilon}\left(s, x\right)}\left\{\int_{s}^{t} \frac{\sqrt{\kappa_r}}{h _{r}(x)} dB_{r}-\left\{\frac{1}{2}+\frac{1}{\bar\kappa}\left(2-\frac{\delta^{2}}{2}-\eta\right)\right\} \int_{s}^{t} \frac{\kappa_r}{h_{r}(x)^{2}} dr\right\}
\]
\[
\Theta_2 := \frac{1}{2 \delta^{2}} \int_{0}^{T} \dot{A}_{r}^{2} d r.
\] 
Now consider \eqref{logh-bound} by putting an equality instead of inequality. That is, let $\psi_t$ be the solution to 
\[ \log(\psi_t) = \log(x) + \Theta - \int_s^t \frac{\eta}{\psi_r^2}dr.\]
Note that this is an ODE for $u_t = \log(\psi_t)$. It can be easily verified that the function $\psi_t =  \sqrt{x^{2} e^{2 \Theta}-2 \eta(t-s)}$ is the solution to this ODE. Hence, using Lemma \ref{Gronwall}, we obtain that 
\begin{equation}
h_{t}(x) \leq \sqrt{x^{2} e^{2 \Theta}-2 \eta(t-s)}.
\end{equation}
Therefore, it follows that \begin{equation}
T_{\varepsilon}(s, x)-s \leq \frac{1}{2\eta}\left(x^{2} e^{2 \Theta}-\varepsilon^{2}\right),
\end{equation}
which implies that 
\begin{equation}\label{Eq7}
\mathbf{E}\left[|T_{\varepsilon}(s, x)-s|^{p}\right]
\lesssim \left(\varepsilon^{2 p}+\mathbf{E}\left[x^{2 p} e^{2 p \Theta}\right]\right).
\end{equation}

\medskip

In order to estimate $\mathbf{E}\left[e^{2 p \Theta}\right]$ we use the same argument as in the proof of Proposition \ref{crucial-estimate}-$(b)$. We again note that 
$\Theta_1$ is dominated by an exponential random variable, and we have assumed that $\Theta_2$ has high enough exponential moments. Given $\bar\kappa<4$, it can be easily checked that we can choose parameters $\delta, \eta$ such that $\mathbf{E}\left[e^{2 p \Theta}\right]$ is bounded for some $p>1$. Then, letting $\epsilon \to 0+$ in \eqref{Eq7} and using monotone convergence theorem, we obtain that
\begin{equation}
\mathbf{E}\left[|T(s, x)-s|^{p}\right] \lesssim  x^{2 p}.
\end{equation}
\end{proof}
We will sometimes need to specify the dependence of $T(s,x)$ on the driving function $U$, and we will write $T(s,x,U)$ to do so. We will also need to consider $T(s,x,M)$, i.e. when $U$ is simply a martingale and $A\equiv 0$. Besides the upper-bound provided by Proposition \ref{Proposition10}, we will also need a lower bound on $\mathbf{E}[T(s,x,M)].$ More generally, note that for any stopping time $\tau$ such that $\tau <\infty$ a.s., by the optional sampling theorem, the process 
$\left\{M_{t}-M_{\tau}\right\}_{t \geqslant \tau}$ is a local martingale satisfying Condition \ref{sem-cond} with $\bar\kappa <4$. Therefore, we can consider the random variable $T(\tau, x, M)$ for any such stopping time $\tau.$ We further claim the following:
\begin{lemma}\label{lower bound}
Let $\tau$ be a stopping time such that $\tau <\infty$ a.s. Then,
\begin{equation}\label{lower-1/4}
\mathbf{E}\left[T(\tau, x, M)-\tau \mid \mathcal{F}_{\tau}\right] \geqslant \frac{x^{2}}{4}.
\end{equation}
\end{lemma}
\begin{proof}
Let $l_t=l_t(x)$ denote the solution to the equation
\begin{equation}
d l_{t}=d M_{t}-\frac{2}{l_{t}} d t, \hspace{2mm} 
l_{\tau}=x, t \geqslant \tau.
\end{equation}
Clearly, since $l_t >0$ for $t <T(\tau, x, M)$,
\begin{equation}\label{l-bound}
l_{t} \leqslant x+M_{t}-M_{\tau} \quad \forall t<T(\tau, x, M).
\end{equation}
Next, consider the process $l_t^2$. By the It\^o formula, 
\begin{equation}\label{l^2-eqn}
l_{t}^{2}=x^{2}+\int_{\tau}^{t} 2 l_{r} d M_{r}+[M]_{t}-[M]_{\tau}-4(t-\tau).
\end{equation} 
Using the optional sampling theorem, conditional on $\mathcal{F}_{\tau}$, the process $
\int_{\tau}^{\cdot} l_{r} d M_{r}$ is a local martingale. We claim that it is in fact a martingale. To this end, observe using \eqref{l-bound} that 
\begin{align*}
\mathbf{E}\left[\left(\int_{\tau}^{t} l_{r} d M_{r}\right)^{2}\right] &=\mathbf{E}\left[\int_{\tau}^{t} l_{r}^{2} d[M]_{r}\right]\\
&\leq \bar\kappa \mathbf{E}\left[\int_{\tau}^{t} l_{r}^{2} d r\right]\\
&\leqslant \bar\kappa \mathbf{E}\left[\int_{\tau}^{t}\left(x+M_{r}-M_{\tau}\right)^{2} d r\right],
\end{align*}
It then follows using the Burkholder-Davis-Gundy inequality that the second moment of $
\int_{\tau}^{t} l_{r} d M_{r}$ is bounded for all bounded $t$, thereby implying that it is a martingale.

Now, note that $
l_{T(\tau, x, M)}=0$. Therefore, using \eqref{l^2-eqn},

\begin{align*}
x^{2}+\int_{\tau}^{T(\tau, x, M)}2{l_{r}} d M_{r}&=
4\left(T(\tau,  x, M)-\tau\right)-\left([M]_{T(\tau,  x, M)}-[M]_{\tau}\right)\\
&\leq 4\left(T(\tau, x, M)-\tau\right).
\end{align*}
Note using Proposition \ref{Proposition10} that $\mathbf{E}\left[T(\tau, x, M)\right]<\infty$. Since $
\int_{\tau}^{\cdot} l_{r} d M_{r}
$ is a martingale w.r.t. to 
$\mathbf{P}\left(\cdot \mid \mathcal{F}_{\tau}\right)$, this implies that 
$$
\mathbf{E}\left[\int_{\tau}^{T(\tau, x, M)} l_{r} d M_{r} |\mathcal{F}_{\tau}\right]=0$$
which implies \eqref{lower-1/4}.
\end{proof}
\vspace{4mm}

The final step towards the simpleness of $\gamma$ is the following proposition.
 
\begin{proposition}\label{Proposition 11}
If U is a semimartingale satisfying Condition \ref{sem-cond} with $\bar\kappa <4$ and \eqref{energy}, then 
$$\mathbf{P}\left[\lim _{x \rightarrow 0+} T(s, x, U)=s \hspace{2mm} \text { for all } s \geqslant 0\right]=1.$$
\end{proposition}

\begin{proof}
We first note that 
$$U_{t}-U_{s}=M_{t}-M_{s}+A_{t}-A_{s} \leqslant M_{t}-M_{s}+\int_{s}^{t}|\dot{A}_{r}| d r.$$
Therefore, if $\tilde{U}_{t}=M_{t}+\int_{0}^{t}|\dot{A}_{r}| d r$, then
$U_{t}-U_{s} \leqslant \tilde{U}_{t}-\tilde{U}_{s}.$
This in turn implies using Lemma \ref{Gronwall} that 
$$T(s, x,U) \leqslant T(s, x, \tilde{U}).$$ 
Therefore, it suffices to prove that simultaneously for all $s\geq 0$, \begin{equation}\label{goal}
\lim _{x \rightarrow 0+} T(s, x, \tilde{U})=s.
\end{equation}
The advantage of switching from $U$ to $\tilde{U}$ is that $\tilde{U}$ has monotonic increasing drift. This will be important in the proof below. For the rest of the proof, we simply write $T(s,x)$ for $T(s, x, \tilde{U})$. Since $\tilde{U}_t$ is constant for $t >T$, it easily follows that $T(s, x)=s+\frac{x^{2}}{4}$ for $s>T$ and \eqref{goal} is trivially true. To prove \eqref{goal} for $s \leq T$, we note using Lemma \ref{Gronwall} that for all $x_1\geq x_2$, we have $ h\left(s, t, x_{1}\right) \geqslant h\left(s, t, x_{2}\right)$. Therefore, $
T(s, x_{1}) \geqslant T(s, x_{2}).$
This implies that $T(s, 0+)-s:=\lim _{x \rightarrow 0+} T(s, x)-s$ exists.
The core of the argument is to prove that this limit is indeed zero for all $s \in [0, T].$ To do so, we note using the flow property of $h(\cdot, \cdot, \cdot)$ that for any $s \leqslant u \leqslant T(s, x)$, 
\begin{equation}\label{chainrelation}
T(s, x)=T(u, h(s, u, x)) \geqslant T(u,0+).
\end{equation}
Therefore, if $T(s, x)-s$ is small, it implies that $T(u, x)-u$ is small for all $s <u < T(s, x)$. We utilize this observation to cover the interval $[0,T]$ using intervals of the form $[s,T(s,x)]$. More precisely, let $x_{n}=2^{-n}.$ Then, for each $n$, define recursively 
$$ s_{0}(n)=0 \quad s_{1}(n)=T(0, x_{n}),$$
$$s_{k+1}(n)=T(s_{k}(n), x_{n}).$$ 
We run this recursion enough number of times (say $K_n$ times) so that $s_{K_n}(n)$ crosses $T.$ Then, the interval $[0,T]$ is covered by union of intervals $\left[s_{k}(n), s_{k+1}(n)\right] .$ To get an estimate of the size of $K_n$, we require a lower bound on the increments $s_{k+1}(n)-s_{k}(n).$ To this end, we rely on Lemma \ref{lower bound} as follows. Define a sequence $m_k$ by $m_0=0$ and
$m_{k}=s_{k}- \sum_{i=1}^{k} \mathbf{E}\left[s_{i}-s_{i-1}| \mathcal{F}_{s_{i-1}}\right] .$ Then, since $\mathbf{E}\left[s_{k}\right]<\infty$ by Proposition \ref{Proposition10}, it easily follows that $m_k$ is a discrete martingale. Furthermore, note that $\tilde{U}_{t}-\tilde{U}_{s}  = M_{t}-M_{s}+\int_{s}^{t}|\dot{A}_{r}| d r \geqslant M_{t}-M_{s}.$ Therefore, by Lemma \ref{Gronwall} 
$$s_{i}-s_{i-1}=T(s_{i-1}, x_{n})-s_{i-1} \geqslant T(s_{i-1}, x_{n}, M)-s_{i-1}.$$
Therefore, using Lemma \ref{lower bound}, 
$$\mathbf{E}\left[s_{i}-s_{i-1} \mid \mathcal{F}_{s_{i-1}}\right] \geqslant \frac{x_{n}^{2}}{4},$$
and it follows that
$$s_{k} \geqslant m_{k}+k \frac{x_{n}^{2}}{4}.$$
Since the martingale $m_k$ is of mean zero, this suggests that to choose $K_n$ such that $s_{K_{n}}>T$ it suffices to take $K_{n}=\left\lfloor\frac{16 T}{x_{n}^{2}}\right\rfloor.$ To make it precise, we claim that for $K_{n}=\left\lfloor\frac{16 T}{x_{n}^{2}}\right\rfloor$,
$$m_{K_{n}} \stackrel{\mathbf P}{\longrightarrow} 0 \text { as, } n \rightarrow \infty .$$

To prove it, note that since $m_k$ is a martingale, using Burkholder-Davis-Gundy inequality 
$$\mathbf{E}\left[m_{K_{n}}^{p}\right]\leq C_{p} \mathbf{E}\left[\left(\sum_{i=1}^{K_{n}}\left(m_{i}-m_{i-1}\right)^{2}\right)^{p/ 2}\right].$$
We choose $p \in (1,2)$ so that the Proposition \ref{Proposition10} is valid. Therefore, 
$$ \mathbf{E}\left[m_{K_{n}}^{p}\right] \leqslant C_{p} \sum_{i=1}^{K_{n}} \mathbf{E}\left[ | m_{i}-m_{i-1}|^{p} \right] .$$
Also, using Proposition \ref{Proposition10} 
\begin{align*}
\mathbf{E}\left[\left |m_{i}-m_{i-1}\right|^{p}\right] &\leqslant C_{p} \mathbf{E}\left[\left|s_{i}-s_{i-1}\right|^{p}\right]\\
&=C_{p} \mathbf{E}\left[\left|T(s_{i-1}, x_{n})-s_{i-1}\right|^{p}\right]\\
&\leqslant C x_{n}^{2 p}.
\end{align*}
This implies that 
$\mathbf{E}\left[m_{K_{n}}^{p}\right] \leqslant C x_{n}^{2 p} K_{n} \rightarrow 0$, as $n \rightarrow \infty.$
This in particular implies that $m_{K_{n}} \stackrel{\mathbf{P}}{\longrightarrow} 0$, as $n \to \infty.$ Since convergence in probability implies a.s. convergence on a subsequence, we conclude that a.s. for infinitely many $n$, we have $ m_{K_{n}}>-T.$ In particular, 
$$ s_{K_{n}} \geqslant m_{K_{n}}+K_{n} \frac{x_{n}^{2}}{4}>T,$$ which implies that almost surely for infinitely many $n$, $[0, T]$ is covered by $\bigcup_{k=0}^{K_{n}-1}\left[s_{k}, s_{k+1}\right] .$
Next, we now establish that each of the increment $s_{k+1}-s_{k} \text { for } 0 \leq k \leq K_{n}-1$ are uniformly small. To this end, let $p>1$ be from Proposition \ref{Proposition10} and consider the event 
$$E_n=\{\text{ For some} \hspace{2mm} 0 \leq k \leq K_n-1, \hspace{1mm} s_{k+1}-s_{k}>2^{-n \frac{p-1}{p}}\}.$$ Then, using Proposition \ref{Proposition10}, 
\begin{align*}
\mathbf{P}\left[E_{n}\right] &\leq K_{n}\mathbf{P}\left[s_{k+1}-s_{k}>2^{-n \frac{p-1}{p}}\right]\\
&\leq K_{n} \frac{E\left[\left|s_{k+1}-s_{k}\right|^{p}\right]}{2^{-n(p-1)}}\\
&\leq C \frac{1}{x_{n}^{2}} \frac{x_{n}^{2 p}}{2^{-n(p-1)}}\\
&=C 2^{-n(p-1)}.
\end{align*}
Therefore, $\sum_{n=1}^{\infty} \mathbf{P}\left[E_{n}\right]<\infty$ and the Borel-Cantelli Lemma implies that almost surely, for all $n$ large enough and for all $k$ with $0 \leq k \leq K_n-1,$ $s_{k+1}(n)-s_{k}(n) \leq 2^{-n(p-1)/p}.$
Finally, for any $u \in [0, T]$, $u \in [s_k(n), s_{k+1}(n)]$ for some $0 \leq k \leq K_n-1,$ for infinitely many $n$. Therefore, by \eqref{chainrelation} 
\begin{align*}
T(u, 0+)-u &\leqslant T(s_{k}(n), x_{n})-u\\
&\leqslant T(s_{k}(n), x_{n})-s_{k}(n)\\
&=s_{k+1}(n)-s_{k}(n)\\
&\leq 2^{-n \frac{(p-1)}{p}}.
\end{align*}
Taking $n \to \infty$, we get $T(u,0+)-u=0$.
\end{proof}

\vspace{2mm}

We are now ready to prove that $\gamma$ is simple. 

\begin{proof}[Proof of the simpleness]
It follows from \eqref{theta exists} that $h(s, t, 0^{\downarrow}) := \lim_{y \to 0+} h(s, t, iy)$ exists, where we have used $0^{\downarrow}$ to emphasise that we are taking vertical limit $iy \to 0$. We first claim that for any $s\in [0,T]$ and $\delta>0$ , $h(s, t, 0^{\downarrow})$ cannot be identically zero for $t \in [s, s+ \delta]$. Suppose the contrary. If for some $\delta>0$, $h(s, t, 0^{\downarrow}) = 0$ for all $t\in [s,s+ \delta]$, then first note from \eqref{theta exists} that 
\[ |h(s, t, iy)|= |h(s, t, iy)- h(s, t, 0^{\downarrow})| \lesssim y^{1-\theta}.\]
\vspace{2mm}

This implies that $X_t(iy)^2 + Y_t(iy)^2 \lesssim y^{2- 2\theta}$, and $Y_t(iy) \lesssim y^{1-\theta}$. Further note from \eqref{Y-eqn} that 
\[ Y_t(iy) = y\exp \biggl(\int_s^t \frac{2}{X_r(iy)^2 + Y_r(iy)^2}dr\biggr).\]
Note that 
\[ \exp \biggl(\int_s^t \frac{2}{X_r(iy)^2 + Y_r(iy)^2}dr\biggr) \gtrsim \exp \biggl(\frac{2(t-s)}{y^{2- 2\theta}} \biggl),\]

This implies that $y^{1-\theta} \gtrsim y \exp \biggl(\frac{2(t-s)}{y^{2- 2\theta}} \biggl)$ which is clearly false for $y$ small enough. This is a contradiction.

\smallskip

We next claim that for all $s< t$, $h(s, t, 0^{\downarrow})  \in \HH$. Otherwise, pick $s<t$ such that $h(s, t, 0^{\downarrow})  \in \R$. Since the backward flow started from a point in $\HH$ stays in $\HH$, this implies that $h(s, u, 0^{\downarrow})  \in \R$ for all $u\leq t$. Since $h(s, \cdot, 0^{\downarrow}) $ is not identically zero, pick a $u_0 \in [s,t]$ such that $h(s, u_0 , 0^{\downarrow}) \neq 0$. Without loss of generality, we may assume that $h(s, u_0, 0^{\downarrow}) > 0$. Now, travelling backward in time from $u=u_0$ to $u=s$, let $\bar{u}$ be the first time where $h(s, \cdot, 0^{\downarrow})$ hits zero (we are kind of choosing a piece of $0$-excursion interval). Now, consider $L_u = h(\bar{u}, u, 0^{\downarrow})$. It follows that $L_{\bar{u}} = 0$ and $L_u > 0$ for all $0< u\leq u_0 $. Also note that 

\begin{equation}\label{L-eqn} 
L_u = U_{u} - U_{\bar{u}} -\int_{\bar{u}}^u \frac{2}{L_r}dr.
\end{equation}
Using comparison of ODE solutions, this implies that for any $x>0$, $h(\bar{u}, u, x) \geq L_u$. This is a contradiction to the fact that $\lim_{x\to 0+} T(\bar{u}, x, U) = \bar{u}$. 

Finally, note using the flow property that for any $s_1 < s_2$,  
\[ h(s_1, T, 0^{\downarrow}) = h(s_2, T, h(s_1, s_2, 0^{\downarrow}) ).\]
Since $h(s_1, s_2, 0^{\downarrow}) \in \HH$, this implies that $h(s_1, T, 0^{\downarrow})  \neq h(s_2, T, 0^{\downarrow})$. This completes the proof by noting from the equation \eqref{f=h} that $ \gamma_t = h(T-t, T, 0^{\downarrow})$. 
\end{proof}

\section{Proof of Theorem \ref{main-thm-2}}  \label{proof-thm-2}

\subsection{Proof of existence of $\gamma$}  \label{proof-thm-2-exists}

We implement the idea explained in Section \ref{sketch-proof}. In the following we denote by $B(z,r)$ the open ball around $z$ with radius $r$, and we denote the conformal radius of $D$ around $z$ by $\crad(z,D)$.

We define a grid of points
\begin{multline}\label{nc:eq:fw_grid}
H(a,M,T) = \biggl\{ x+iy \mid x = \pm aj/8,\ y = a(1+k/8),\ j,k \in \NN \cup \{0\},\\
\abs{x} \le M,\ y \le \sqrt{1+4T} \biggr\} .
\end{multline}
This grid is chosen so that we have $\dist(z,H(a,M,T)) < \frac{a}{8}$ for every $z \in [-M,M] \times [a,\sqrt{1+4T}]$.

A consequence of Koebe's distortion theorem is that the (forward) Loewner flows started from this grid of points already contain all the information on the derivative of the conformal maps $\hat f_t = g_t^{-1}(W_t+\cdot)$. The following lemma is purely deterministic and holds for any continuous driving function $W$. We write $\norm{W}_{\infty, [0,T]}$ for the supremum of $\abs{W}$ on the interval $[0,T]$. Recall the definition of $Z_t(z) = R_t + iI_t$ from \eqref{RI-def}.

\begin{lemma}[See Lemma~2.3 in \cite{yuan-psivarx} for a proof]\label{nc:thm:fw_grid}
Let $\delta \in {]0,1]}$, $u > 0$ and suppose $|\hat f_t'(i\delta)| \ge u$ for some $t \in [0,T]$. Then there exists $z \in H(u\delta,\norm{W}_{\infty;[0,T]},T)$ such that
\[ \quad |Z_{t}(z)-i\delta| \le \frac{\delta}{2} \quad \text{and} \quad |g_{t}'(z)| \le \frac{80}{27}\,\frac{1}{u}, \]
where $H(a,M,T)$ is given by \eqref{nc:eq:fw_grid}.
\end{lemma}

\begin{remark}
For later reference, let us note here that the condition $|Z_{t}(z)-i\delta| \le \delta/2$ implies in particular
\[
I_t(z) \in [\delta/2, 3\delta/2] \quad \text{and} \quad \abs*{\frac{R_t(z)}{I_t(z)}} \le 1 .
\]
\end{remark}

Next, we introduce the parametrisation by imaginary value. For $z \in \HH$ and $\delta > 0$, let $\sigma(s) = \sigma(s,z,\delta) = \inf\{ t \in \RR \mid I_t \le \delta e^{-2s} \}$, $s \in \RR$. Note that the $s$-parametrisation is defined such that the flow starts at $s_0 = -\frac{1}{2}\log\frac{y}{\delta}$, while $s=0$ corresponds to the time $t$ when $I_t(z)=\delta$. 
We have the following representations
\[
\sigma(s) = \int_{-\frac{1}{2}\log\frac{y}{\delta}}^s (R_{\sigma(s')}^2+I_{\sigma(s')}^2)\,ds' 
\]
and
\[ \partial_s \log |g_{\sigma(s)}'(z)| =  -2 \frac{R_{\sigma(s)}^2-I_{\sigma(s)}^2}{R_{\sigma(s)}^2+I_{\sigma(s)}^2}, \]
and consequently $\abs*{\partial_s \log |g_{\sigma(s)}'(z)|} \le 2$.

\bigskip

Now we consider $W_t = M_t+A_t$ as in Condition \ref{sem-cond}. Let us first consider the case when drift term $A \equiv 0$. We therefore assume that $W_t = \int_0^t \sqrt{\kappa_s} \, dB_s$ for some adapted process $\kappa_s$. 
The moments of $|g_t'(z)|$ can then be studied similarly to the case of the backward flow.

For $F = F(w,y) \in C^2$ we have
\begin{equation}\label{nc:eq:fw_M_ito}
d \left( \abs{g_t'(z)}^\lambda I_t^{-\lambda} F(\frac{R_t}{I_t},I_t) \right) 
= \abs{g_t'(z)}^\lambda I_t^{-\lambda-2} \big( \Lambda_{\kappa_t}F \, dt -\sqrt{\kappa_t}I_t F_w \, dB_t \big) .
\end{equation}
where
\[ \Lambda_\kappa F = \Lambda^{\mathrm{(fw)}}_\kappa F \defeq \frac{4\lambda}{(1+w^2)^2}F -\frac{2y}{1+w^2}F_y +\frac{4w}{1+w^2}F_w +\frac{\kappa}{2}F_{ww} . \]

\begin{lemma}
The function $F(w,y) = (1+w^2)^{r/2} y^{\zeta+\lambda}$ satisfies $\Lambda^{\mathrm{(fw)}}_\kappa F \le 0$ on $\HH$ if and only if $\lambda-\zeta \le -\frac{r\kappa}{4}$ and $\lambda+\zeta \ge 2r-\frac{r\kappa}{4}+\frac{r^2 \kappa}{4}$. 
\end{lemma}

\begin{remark}\label{nc:rm:fw_smarter_supersol}
Here again the regularity of $\gamma$ that can be proven is directly related to the exponents $\lambda,\zeta$ (with some restrictions on $r$). So one may again ask for sharper supersolutions. In contrast to the backward case, we had to modify the exponents in $F$ also in the case $\bar\kappa < 8$, so optimal regularity of $\gamma$ is not clear in that case either. We believe that its regularity should be the same as for SLE$_{\kappa_*}$ where $\kappa_* = \bar\kappa$ in the case $\bar\kappa < 8$, and $\kappa_* = \ubar\kappa$ in the case $\ubar\kappa > 8$.

Under certain conditions on $\kappa_*$, $\bar\kappa-\ubar\kappa$, and $r$, we can find supersolutions to the equation $\sup_{\kappa \in [\ubar\kappa,\bar\kappa]} \Lambda_\kappa F \le 0$ that are of the form
\[ F(w,y) = y^{\zeta+\lambda}(1+w^2)^{r/2}\exp(g(w)) \]
with $\lambda,\zeta$ chosen according to Remark \ref{nc:rm:fw_parameters} with $\kappa = \kappa_*$ and a bounded function $g$. More precisely, we can pick $g$ such that $g'$ is of the form
\[ g'(w) = 
\begin{cases}
-\alpha_1 w & \text{for } w \le w_0,\\
-\alpha_2 w^{-1-\varepsilon} & \text{for } w \ge w_0.
\end{cases}
\]
This works whenever $\bar\kappa-\ubar\kappa$ is sufficiently small (depending on $\kappa_*,r$). Again, we did not succeed in making this work in general.
\end{remark}

\begin{corollary}\label{nc:thm:fw_supermartingale}
Suppose $A \equiv 0$ and $\kappa_t \in [\ubar\kappa,\bar\kappa]$ for all $t$. Let $r,\lambda,\zeta$ be chosen such that $\Lambda_\kappa F \le 0$ for all $\kappa \in [\ubar\kappa,\bar\kappa]$. Then the process
\[ M_t = |g_t'(z)|^\lambda I_t^\zeta (1+R_t^2/I_t^2)^{r/2}, \quad t \ge 0 \]
is a supermartingale.
\end{corollary}

\begin{remark}\label{nc:rm:fw_parameters}
In case of constant $\kappa$, i.e. $\ubar\kappa = \bar\kappa$, we can take
\[ \begin{split}\label{nc:eq:fw_constant_kappa_parameters}
\lambda &= r-\frac{r\kappa}{4}+\frac{r^2\kappa}{8} ,\\
\zeta &= r+\frac{r^2\kappa}{8} .
\end{split} \]
In that case, $\Lambda_\kappa F = 0$ and $(M_t)$ is a martingale when stopped before the hull hits a small ball around $z$.
\end{remark}

\begin{proof}
Let $F(w,y)=(1+w^2)^{r/2} y^{\zeta+\lambda}$ as above. Then $M_t =\abs{g_t'(z)}^\lambda I_t^{-\lambda} F(R_t/I_t,I_t)$. By \eqref{nc:eq:fw_M_ito} and our assumption $\Lambda_{\kappa_t}F \le 0$, we have that $(M_t)$ is a non-negative local supermartingale, and therefore a supermartingale.
\end{proof}

Recall that by \cref{nc:thm:fw_grid}, if $\abs{\hat f_t'(i\delta)} \ge u$ for some $t \in [0,T]$, then we can find $z \in H(u\delta,\norm{W}_{\infty;[0,T]},T)$ that satisfies the property stated in that lemma. Note that for such $z$, we have $\sigma(s,z,\delta) = t$ for some $s \in [-1,1]$. In particular, $\abs{g_{\sigma(s)}'(z)} \lesssim \frac{1}{u}$ and $\abs*{\frac{R_{\sigma(s)}}{I_{\sigma(s)}}} \le 1$ for some $s \in [-1,1]$.

In case $\lambda < 0$, a lower bound for $|g_t'(z)|$ is equivalent to an upper bound for $|g_t'(z)|^\lambda$. Then
\[ \pr\left( |g_{\sigma(s)}'(z)| \le \frac{1}{u} \text{ and } |R_{\sigma(s)}| \le I_{\sigma(s)} \right) \le u^{\lambda} \ex\left[ |g_{\sigma(s)}'(z)|^\lambda 1_{|R_{\sigma(s)}| \le I_{\sigma(s)}} \right]  \]
for fixed $s$. Moreover, since $\abs*{\partial_s \log |g_{\sigma(s)}'(z)|} \le 2$, we have $\dfrac{|g_{\sigma(s)}'(z)|}{|g_{\sigma(0)}'(z)|} \in [e^{-2}, e^2]$ for all $s \in [-1,1]$.

Let $S = S(z,\delta) = \inf\{ s \in [-1,1] \mid |R_{\sigma(s)}| \le I_{\sigma(s)} \} \wedge 2$. Together with \cref{nc:thm:fw_supermartingale}, we then have (for any $\lambda \in \RR$)
\begin{align}
\ex\left[ |g_{\sigma(S)}'(z)|^\lambda 1_{S \le 1} \right] 
&\asymp \delta^{-\zeta} \ex\left[ M_{\sigma(S)} 1_{S \le 1} \right] \nonumber\\
&\le \delta^{-\zeta} M_0 \nonumber\\
&\le \delta^{-\zeta} y^\zeta (1+x^2/y^2)^{r/2} \label{nc:eq:fw_gprime_moment}
\end{align}
and consequently (for $\lambda \le 0$)
\begin{align}
&\pr\left( |g_{\sigma(s)}'(z)| \le \frac{1}{u} \text{ and } |R_{\sigma(s)}| \le I_{\sigma(s)} \text{ for some } s \in [-1,1] \right) \nonumber\\
&\quad \lesssim u^\lambda \ex\left[ |g_{\sigma(S)}'(z)|^\lambda 1_{S \le 1} \right] \nonumber\\
&\quad \lesssim u^\lambda \delta^{-\zeta} y^\zeta (1+x^2/y^2)^{r/2} . \label{nc:eq:fw_gridpoint_probability}
\end{align}

\begin{proposition}\label{nc:thm:fw_fprime_probability}
Suppose $r,\lambda,\zeta$ are chosen according to \cref{nc:thm:fw_supermartingale} and $\lambda \le 0$. Then there exists a constant $C < \infty$, depending on $r,\zeta,\lambda,T,M$, such that for $\delta, u > 0$ we have
\begin{multline*}
    \pr ( |\hat f_t'(i\delta)| \ge u \text{ for some } t \in [0,T] ,\ \norm{W}_{\infty, [0,T]} \le M) \\
    \le \begin{cases}
    C u^{\zeta+\lambda} & \text{if } r < -1,\ \zeta+1 < -1\\
    C u^{\lambda-2} \delta^{-\zeta-2} & \text{if } r < -1,\ \zeta+1 > -1,\\
    C u^{\zeta+\lambda-(r+1)} \delta^{-(r+1)} & \text{if } r > -1,\ \zeta-r < -1\\
    C u^{\lambda-2} \delta^{-\zeta-2} & \text{if } r > -1,\ \zeta-r > -1.
    \end{cases}
\end{multline*}
\end{proposition}

\begin{proof}
With \cref{nc:thm:fw_grid}, we only need to sum up \eqref{nc:eq:fw_gridpoint_probability} for all points $z \in H(u\delta,M,T)$. The result follows from following \cref{nc:le:fw_sum_grid}. 
\end{proof}

\begin{lemma}[See Lemma~2.6 in \cite{yuan-psivarx} for a proof]\label{nc:le:fw_sum_grid}
Let $r,\zeta \in \RR$, $M,T > 0$, $a > 0$. Then there exists $C < \infty$ depending on $r,\zeta,M,T$ such that
\[ \sum_{z \in H(a,M,T)} y^\zeta(1+x^2/y^2)^{r/2} \le 
\begin{cases}
Ca^\zeta & \text{if } r < -1,\ \zeta+1 < -1,\\
Ca^{-2} & \text{if } r < -1,\ \zeta+1 > -1,\\
Ca^{\zeta-r-1} & \text{if } r > -1,\ \zeta-r < -1,\\
Ca^{-2} & \text{if } r > -1,\ \zeta-r > -1.
\end{cases} \]
\end{lemma}

\begin{corollary}
Suppose $r,\lambda,\zeta$ are chosen according to \cref{nc:thm:fw_supermartingale} and $\lambda \le 0$. Let $\beta > \frac{\zeta+2}{2-\lambda} \vee \frac{r+1}{r+1-\zeta-\lambda} \vee 0$. Then with probability $1$ there exists some (random) $y_0 > 0$ such that
\[ |\hat f_t'(i\delta)| \le \delta^{-\beta} \]
for all $\delta \in {]0,y_0]}$ and $t \in [0,T]$.
\end{corollary}

\begin{proof}
It suffices to show the claim on the event $\{\norm{W}_{\infty, [0,T]} \le M\}$ for all $M$. By \cref{nc:thm:fw_fprime_probability}
\begin{multline*}
    \pr ( |\hat f_t'(i\delta)| \ge \delta^{-\beta} \text{ for some } t \in [0,T] ,\ \norm{W}_{[0,T]} \le M) \\
    \le \begin{cases}
    C \delta^{-\beta(\zeta+\lambda)} & \text{if } r < -1,\ \zeta+1 < -1\\
    C \delta^{-\beta(\lambda-2)-\zeta-2} & \text{if } r < -1,\ \zeta+1 > -1,\\
    C \delta^{-\beta(\zeta+\lambda-(r+1))-(r+1)} & \text{if } r > -1,\ \zeta-r < -1\\
    C \delta^{-\beta(\lambda-2)-\zeta-2} & \text{if } r > -1,\ \zeta-r > -1.
    \end{cases}
\end{multline*}
Our choice of $\beta$ implies that this probability decays as $\delta \searrow 0$.

For $\delta = 2^{-n}$, $n \to \infty$, the claim then follows from the Borel-Cantelli lemma, and for all other $\delta$ from the Koebe distortion theorem.
\end{proof}

\begin{proof}[Proof of Theorem \ref{main-thm-2}]
When the drift $A\equiv 0$, if we can pick $\beta < 1$ in the previous corollary, then by \cite[Corollary 3.12]{vl-sle-hoelder} the trace exists. This is possible if and only if $\frac{\zeta+2}{2-\lambda} < 1 \iff \zeta+\lambda < 0$.

For better readability, we write down the two cases $0 = \ubar\kappa \le \bar\kappa < 8$ and $8 < \ubar\kappa \le \bar\kappa < \infty$ separately.

First the case $\bar\kappa < 8$. In order to fulfill also the conditions of \cref{nc:thm:fw_supermartingale}, we need to pick $r$ such that $2r-\frac{r\kappa}{4}+\frac{r^2 \kappa}{4} < 0$ for all $\kappa \in [0,\bar\kappa] \iff r \in {\left] 1-\frac{8}{\bar\kappa}, 0 \right[}$. This is a non-empty interval if and only if $\bar\kappa < 8$.

Next, we need to fulfill $\lambda-\zeta \le -\frac{r\kappa}{4}$. Since we allow $\kappa$ to be as small as $0$, and this condition becomes $\lambda-\zeta \le 0$.

In summary, we need to pick $\zeta$, $\lambda$ such that $\lambda \le 0$, $\zeta+\lambda \in {\left[ 2r-\frac{r\bar\kappa}{4}+\frac{r^2 \bar\kappa}{4}, 0 \right[}$, and $\zeta-\lambda \ge 0$. This can be done by choosing $\zeta = \lambda = r-\frac{r\bar\kappa}{8}+\frac{r^2 \bar\kappa}{8}$.

Now the case $\ubar\kappa > 8$. Again, we need to pick $r$ such that $2r-\frac{r\kappa}{4}+\frac{r^2 \kappa}{4} < 0$ for all $\kappa \in [\ubar\kappa,\bar\kappa] \iff r \in {\left] 0, 1-\frac{8}{\ubar\kappa} \right[}$. This is a non-empty interval if and only if $\ubar\kappa > 8$.

The condition $\lambda-\zeta \le -\frac{r\kappa}{4}$ for all $\kappa \in [\ubar\kappa,\bar\kappa]$ now becomes $\lambda-\zeta \le -\frac{r\bar\kappa}{4}$. 

In summary, we need to pick $\zeta$, $\lambda$ such that $\lambda \le 0$, $\zeta+\lambda \in {\left[ 2r-\frac{r\ubar\kappa}{4}+\frac{r^2 \ubar\kappa}{4}, 0 \right[}$, and $\zeta-\lambda \ge \frac{r\bar\kappa}{4}$. This can be done by choosing $\zeta = r-\frac{r\ubar\kappa}{8}+\frac{r\bar\kappa}{8}+\frac{r^2 \ubar\kappa}{8}$, $\lambda = r-\frac{r\ubar\kappa}{8}-\frac{r\bar\kappa}{8}+\frac{r^2 \ubar\kappa}{8}$. 

\smallskip

For non-zero $A$, the above argument remains valid as long as in \eqref{nc:eq:fw_gprime_moment}, we verify the estimate $\E[M_{\sigma(S)}] \lesssim M_0$ for $r, \lambda, \zeta$ chosen as above. This can be verified similarly as in the proof of Theorem \ref{main-thm-1}. By the same argument, we can assume \eqref{energy} holds. We then just need to note that, similarly as \eqref{h'=exp}, the equation \eqref{nc:eq:fw_M_ito} also yields an identity for $\abs{g_t'(z)}^\lambda I_t^{-\lambda} F(\frac{R_t}{I_t},I_t)$ in terms of exponential of a martingale plus some an additional integral depending on $A$. At the cost of slightly changing parameters $r,\lambda, \zeta$, this additional integral can be handled similarly as in the proof of Theorem \ref{main-thm-1}. In view of being repetitive, we leave the details to interested readers. 
\end{proof}

\subsection{Proof of simpleness of $\gamma$}  \label{proof-thm-2-simple}

The proof of simpleness of $\gamma$ for the forward case is similar to that of \cite{rs-sle}. Following the steps of the proof of \cite[Theorem 6.1]{rs-sle}, we note that it is sufficient to prove that if $Z_t(x)$ denote the forward flow started at $x >0$, then $Z_t(x) >0$ for all time $t>0$. More precisely, let 
\[ L^x = \inf\{t>0 \mid Z_t(x) = 0\}.\]

We then claim the following. 

\begin{proposition}
If $W$ is semimartingale satisfying Condition \ref{sem-cond} either with $\bar\kappa \leq 4$ and $A\equiv 0$ or with $\bar\kappa <4$ and non-zero $A$ satisfying \eqref{energy}, then $\P[ L^x = +\infty] =1$.
\end{proposition}

\begin{proof}
The proof is a simple adaptation of the proof for Bessel processes. For $\epsilon < x< M$, let $L_\epsilon = \inf\{t>0 \mid Z_t(x) = \epsilon\}, $ and $L_M = \inf\{t>0 \mid Z_t(x) = M\}$. 

\smallskip

Let us first consider the case $A\equiv 0$ and $\bar\kappa \leq 4$. Applying It\^o formula to $\log(Z_t)$, note that 
\begin{equation}\label{sub-mart}
d\log(Z_t) = \frac{1}{Z_t}dW_t + (2 - \frac{\kappa_t}{2}) \frac{1}{Z_t^2}dt. 
\end{equation}

Since $\bar\kappa \leq 4$, this implies that $\log(Z_t)$ is a submartingale for $t\leq L_\epsilon \wedge L_M$. It therefore implies that $\E[\log (Z_{L_\epsilon \wedge L_M})] \geq \log (x)$. Then, a simple computation implies that 
\[ \P[ L_\epsilon  > L_M] \geq \frac{\log(x)- \log(\epsilon)}{\log(M)- \log(\epsilon)}.\]
Letting $\epsilon \to 0+$ and $M \to \infty$ proves the claim. 

\smallskip

For $\bar\kappa <4$ with non-zero $A$, note again using It\^o formula that 
\[ d\log(Z_t) + b \dot{A}_t^2dt = \frac{1}{Z_t}dW_t + \frac{1}{Z_t}\dot{A}_tdt + (2 - \frac{\kappa_t}{2}) \frac{1}{Z_t^2}dt + b \dot{A}_t^2dt.\]
Since $\bar\kappa <4$, we can choose $b$ large enough such that 
\[ \frac{1}{Z_t}\dot{A}_t + (2 - \frac{\kappa_t}{2}) \frac{1}{Z_t^2} + b \dot{A}_t^2 \geq 0 \]  
by using $\abs*{\frac{1}{Z_t}\dot{A}_t} \le \varepsilon\frac{1}{Z_t^2}+\frac{1}{4\varepsilon}\dot{A}_t^2$ with some small $\varepsilon > 0$.

Therefore, $\log(Z_t) + b \int_0^t \dot{A}_r^2dr$ is a submartingale. Noting that $\E[\int_0^T \dot{A}_r^2 dr ] < \infty$, repeating the same argument as above completes the proof. 
\end{proof}

\begin{proof}[Proof of simpleness of $\gamma$]
Knowing the above proposition, the rest of the proof is similar to the one in \cite{rs-sle}. We therefore only give a brief sketch. Suppose for some $t_1< t_2$, $\gamma_{t_1}= \gamma_{t_2}$. We then pick a rational point $s \in (t_1, t_2)$ and apply the mapping out function $g_s(\cdot)$. Note that the mapped out family of compact hulls is driven by $W_{t+s} - W_s$, which is also a semimartingale satisfying condition \ref{sem-cond}. Therefore, the mapped out compact hulls are also generated by a curve, call it $\gamma^s_t$. Furthermore, by the assumption $\gamma_{t_1}= \gamma_{t_2}$, $\gamma^s_{t_2-s} \in \R$. But this implies that the forward flow started at $\gamma^s_{t_2-s}$ hits zero in finite time. This is a contradiction to the above proposition. 
\end{proof}

\section{Proof of Corollary \ref{imp-cor}} \label{proof-imp-cor}

The Corollary \ref{imp-cor} follows at once if we verify that for $\alpha> 3/2$ and $t$ small enough, $|B_t|^\alpha$ satisfies the condition in Theorem \ref{main-thm-1}, \ref{main-thm-2} required for the existence of a simple $\gamma$. To this end, note that for $\alpha>1$, using the It\^o-Tanaka formula,
\begin{equation}
|B_t|^\alpha = \int_0^t \alpha\hspace{1mm} \mbox{sgn}(B_s)|B_s|^{\alpha -1}dB_s + \frac{\alpha(\alpha-1)}{2} \int_0^t |B_s|^{\alpha -2}ds. 
\end{equation}
Also note using the occupation times formula that 
\[ \int_0^t |B_s|^{\alpha -2}ds = \int_{-\infty}^\infty \frac{L_t^a}{|a|^{2-\alpha}}da,\]
where $L_t^a$ is the Brownian local time. It follows that $\int_0^t |B_s|^{\alpha -2}ds < \infty$ a.s. if and only if \footnote{$\int_0^t B_s^{\alpha -2}ds$ is not absolutely convergent for $\alpha \leq 1$. One can however choose a principal value for $\alpha \in (1/2, 1]$, see \cite[p. 236]{ry-stochastic-book} for details.} $\alpha>1$. Similarly, if $A_t =  \int_0^t |B_s|^{\alpha -2}ds$, then 
\[ \int_0^T \dot{A}_s^2 ds =  \int_0^T |B_s|^{2(\alpha -2)}ds < \infty \hspace{2mm} \mbox{a.s.}\] if and only if $2(2-\alpha) < 1$, or equivalently $\alpha > 3/2$.

Let $M_t = \int_0^t \alpha\hspace{1mm} \mbox{sgn}(B_s)|B_s|^{\alpha -1}dB_s$. For $\theta >0$, let $\tau_\theta = \inf \bigl\{t \geq 0\bigm | (\alpha |B_t|^{\alpha-1})^2 \geq \theta\bigr\}$. Then for $\theta < 8$, the Condition~\ref{sem-cond} is satisfied for $M_{t \wedge \tau_\theta}$. Similarly, for $\theta > 4$, the generated curve is simple. Hence, Corollary \ref{imp-cor} follows.

\section{Proof of Theorem \ref{main-thm-3}}  \label{proof-thm-3}

Theorem \ref{main-thm-3} is an easy consequence of the techniques developed above. For the existence of a solution $\varphi_t$, note that \eqref{theta exists} easily implies that $h(s,t,0^{\downarrow}) := \lim_{y\to 0+} h(s,t,iy)$ exists a.s. uniformly in $s,t$. 
Clearly, using the It\^o formula, 
\[ h(s,t,iy)^2 = -y^2 + \int_s^t 2h(s,r,iy) dV_r + \int_s^t (\kappa_r - 4)dr.\]

Letting $y\to 0+$, it easily follows that $\varphi_t = h(s,t,0^{\downarrow})^2$ is a solution to \eqref{Eq2}. 

\smallskip

For the uniqueness, let $\varphi_t$ be any solution to \eqref{Eq2} with some choice of a branch square root $\sqrt{\varphi_t}^b$. Let $\sqrt{\varphi_t}^b = X_t + i Y_t$. It then follows that 
\begin{align}
X_t^2 - Y_t^2 &= 2 \int_0^t X_r dV_r + \int_0^t (\kappa_r - 4)dr, \\
X_tY_t &= \int_0^t Y_r dV_r.
\end{align}

Let $\tau = \inf\{ t>0 \mid \varphi_t \notin [0, \infty)\}$. We want to prove that $\tau =0$ almost surely.  Suppose $\P(\tau >0) > 0$. 
Note that $Y_t =0$ for $t\leq \tau$, which in turn implies 
\begin{equation}\label{when Y=0}
X_t^2 = 2\int_0^t X_r dV_r + \int_0^t (\kappa_r - 4)dr
\end{equation}
for all $t\leq \tau$. Using It\^o formula, for any $\epsilon >0$, 

\[\sqrt{X_t^2 + \epsilon} = \sqrt{\epsilon} + \int_0^t \frac{X_r}{\sqrt{X_r^2 + \epsilon}}dV_r + \int_0^t \frac{\epsilon \kappa_r - 4\epsilon - 4X_r^2}{2(X_r^2 + \epsilon)^{3/2}}dr.\]

As $\epsilon \to 0+$, the term $\int_0^t \frac{X_r}{\sqrt{X_r^2 + \epsilon}}dV_r$ converges uniformly in probability to $\tilde{V}_t := \int_0^t \mbox{sgn}(X_r) 1_{\{X_r \neq 0\}}dV_r$ which is also a semimartingale satisfying condition \ref{sem-cond}. Since $\kappa_t \leq \bar\kappa <4$, note from \eqref{when Y=0} that $X$ cannot be identically zero on $[0,\tau]$. We can therefore pick an interval $[u,v] \subset[0,\tau]$ such that $X_u =0$ and $|X_t| >0$ for all $t\in (u,v]$ (a piece of a $0$-excursion). It then follows that 
\[ |X_t| = \tilde{V}_t - \tilde{V}_u - \int_u^t \frac{2}{|X_r|}dr.\]
Similarly as in the proof of simpleness in Theorem \ref{main-thm-1}, this is a contradiction to Proposition \ref{Proposition 11}. It therefore implies $\tau=0$ a.s. and $Y_t >0$ for all $t>0$, noting that $Y_t > 0$ implies $Y_s > 0$ for all $s>t$.

\smallskip

To prove the uniqueness of $\varphi$, we first claim that for some large enough $L$, there exists a sequence $t_n \to 0+$ such that $|X_{t_n}| \leq LY_{t_n}$. On the contrary, suppose $|X_{t}| > LY_{t}$ for all $t$ small enough. Without loss of generality, assume $X_{t} > LY_{t}$. It then follows that 
\[X_t - X_{\epsilon}= V_t - V_{\epsilon} - \int_\epsilon^t \frac{2X_r}{X_r^2 + Y_2^2}dr.\]
Since $X$ is positive, letting $\epsilon \to 0+$ implies 
\[X_t = V_t - \int_0^t \frac{2X_r}{X_r^2 + Y_2^2}dr \leq V_t - \frac{2L^2}{L^2 + 1}\int_0^t \frac{1}{X_r}dr.\]

Note that Proposition \ref{Proposition 11} remains valid if the constant $-2$ in the backward LDE is replaced by some constant close enough to $-2$. Therefore, if we pick $L$ large enough, using the Gr\"onwall inequality Lemma \ref{Gronwall}, this leads to a contradiction to Proposition \ref{Proposition 11} similarly as above. Finally, to establish the uniqueness of $\varphi$, note that $\sqrt{\varphi_t}= h(t_n, t, \sqrt{\varphi_{t_n}})$.  The estimate \eqref{theta exists} implies that $h(s,t, iy)$ converges uniformly in $s,t$ to $h(s,t,0^{\downarrow})$ as $y\to 0+$. It follows using the distortion theorem that $h(s,t, z)$ also converges uniformly in $s,t$ to $h(s,t,0^{\downarrow})$ as $z\to 0$ non-tangentially with $z\in \{|x|\leq Ly\}$. It therefore follows that $\varphi_t = \lim_{(s,y) \to (0+, 0+)} h(s,t, iy)^2$ which implies the uniqueness of $\varphi$.

\section{Applications} \label{app-sec}

Beside the case of SLEs with Brownian motion as drivers, Loewner chain driven by various non-Brownian drivers have many times appeared in the literature, see e.g.  \cite{cr-le-symmetric-stable} with symmetric $\alpha$-stable drivers in relation to extremal domains of integral means spectrum, \cite{cardy, healey-lawler} with Dyson Brownian motion drivers in relation to $N$-sided radial SLE. Semimartingale drivers are also natural and relevant. In this section we discuss an application of our main results.  
 
\subsection{An Application to stochastic Komatu-Loewner evolutions (SKLEs)}
SKLEs are variants of SLEs in finitely connected domains. For readers' convenience, we start by giving a brief introduction to SKLEs, see \cite{chen-book} and references therein for a detailed account on this subject.

The study of Komatu-Loewner evolutions was initiated by Y. Komatu \cite{komatu} for obtaining a variant of Loewner differential equation in circular slit annuli domains. The corresponding chordal variant in standard slit domains (which are prototypical finitely connected domains) was studied by Bauer and Friedrich \cite{bauer-1}. Chen-Fukushima-Rohde \cite{chen-fukushima-rohde} introduced \textit{Brownian motion with darning (BMD)} to describe the right hand side of the Komatu-Loewner equation as a complex Poisson kernel and established Loewner-Komatu equation as an actual ODE.

A standard slit domain is a domain of the form $\mathbf{H} \setminus \cup_{k=1}^N C_k$, where $C_k$, $1 \leq k\leq N$, are  mutually disjoint horizontal line segments in $\mathbf{H}$. A standard slit domain is characterised by left and right endpoints of its slits which can represented by a $3N$ tuple $\boldsymbol{s} = (\boldsymbol{y}, \boldsymbol{x}, \boldsymbol{x}^r)$,  where $\boldsymbol{y}$ are the heights of the slits and $\boldsymbol{x}, \boldsymbol{x}^r$ contain the real parts of the left and right endpoints of slits, respectively. Let $D$ be a standard slit domain and $\gamma\colon [0,T] \to \overline{D}$ be a simple curve with $\gamma(0) \in \mathbf{R}$ and $\gamma(0,T] \subset D$. For $t\in [0,T]$, let $g_t$ be the unique conformal map from $D\setminus \gamma(0,t]$ onto a standard slit domain $D_t$ satisfying 
\[
g_t(z) = z + \frac{a_t}{z} + o\biggl(\frac{1}{|z|}\biggr) \textrm{ as $z \to \infty$}.
\]
Assuming $a_t =2t$ (which can be obtained via reparametrisation similarly as in chordal Loewner theory), the Komatu-Loewner differential equation describes the evolution of maps $g_t(z)$ and is given by 
\begin{equation}\label{KLE-1}
\frac{dg_t(z)}{dt} = -2\pi \Psi_{D_t }(g_t(z), \xi_t), \textrm{ $g_0(z) = z$},
\end{equation}
where $\Psi_{D_t}(z,\xi)$ is the complex Poisson kernel for the Brownian motion with darning for $D_t$ and $\xi_t = \lim_{z \to \gamma_t }g_t(z)$. The evolution of slits of domains $D_t$ which are encoded by $3N$ tuples $\boldsymbol{s}(t)$ is given by
\begin{equation}\label{KLE-2} 
\frac{d\boldsymbol{s}_j(t)}{dt} = b_j(\boldsymbol{s}(t) - \hat{\xi}_t),
\end{equation}
where $b_j$ are as defined in \cite[equation~(1.5)]{chen-fukushima-suzuki},  and the hat symbol represents horizontal translation of slits, i.e.
\[
(\boldsymbol{y}, \boldsymbol{x}, \boldsymbol{x}^r) - \hat\xi \defeq (\boldsymbol{y}, \boldsymbol{x}-\xi, \boldsymbol{x}^r-\xi) .
\]
Conversely, given any continuous real valued function $\xi_t$, we can solve \eqref{KLE-2} and then \eqref{KLE-1}. It can be shown that for $z\in D$, \eqref{KLE-1} admits a unique solution $g_t(z)$ defined up to a maximal time $t< T_z$. Furthermore, for $F_t := \{ z\in D \mid T_z \leq t \}$, the map $g_t\colon D\setminus F_t \to \mathbf{H}\setminus \boldsymbol{s}(t)$ is a conformal bijection.

To define a variant of SLE in standard slit domains $D$, it is natural to look for families $\{F_t\}_{t \in [0,T]}$ that satisfy \textit{domain Markov property} and \textit{conformal invariance}. Chen-Fukushima \cite{Chen-Fukushima} identified such families by showing that (under certain regularity assumptions) the driving process $\boldsymbol{W}_t = (\xi_t, \boldsymbol{s}(t))$ must be given by solution to the SDE 
\begin{alignat*}{2}
d\xi_t &= \alpha(\boldsymbol{s}(t) - \hat{\xi}_t)dB_t + b(\boldsymbol{s}(t) - \hat{\xi}_t)dt,\\
d\boldsymbol{s}_j(t) &= b_j(\boldsymbol{s}(t) - \hat{\xi}_t)dt,& 1\leq j\leq 3N,
\end{alignat*}
where $B$ is standard Brownian motion and $\alpha(\boldsymbol{s}), b(\boldsymbol{s})$ are locally Lipschitz continuous and homogenous functions of degree $0,-1$, respectively. The family of conformal maps $\{g_t\}_{t\in [0,T]}$ driven by such $\xi_t$ is defined to be SKLE$_{\alpha, b}$. Also see the related work \cite{dapeng-SLE-1} for variants of SLEs in annulus.

Loewner chains driven by semimartingales appear naturally while studying SKLE$_{\alpha, b}$. It was shown in \cite{chen-fukushima-suzuki} that, after a suitable reparametrisation, SKLE$_{\alpha,b}$ has the same distribution as that of a Loewner chain driven by a semimartingale. More precisely, for the SKLE$_{\alpha, b}$ $\{F_t\}_{t\geq 0}$, if we consider the Riemann map $g_t^0\colon \mathbf{H}\setminus F_t \to \mathbf{H}$ with $g_t^0(z) \sim z$ as $z\to \infty$, then 
\[ \frac{dg_t^0(z)}{dt} = \frac{2(\Phi_t'(\xi_t))^2}{g_t^0(z) - W_t},\]
where $\Phi_t = g_t^0\circ g_t^{-1}$ and $W_t = \Phi_t(\xi_t)$ ($\Phi_t$ admits an analytic extension to $D_t \cup \overline{D}_t \cup \partial \mathbf{H}$, see \cite[Lemma~2.1]{chen-fukushima-suzuki}). Moreover, $W_t$ admits a semimartingale decomposition as 
\begin{equation}\begin{split}\label{sem-decop} dW_t = \Phi_t'(\xi_t)\alpha(\boldsymbol{s}(t) &- \hat{\xi}_t)dB_t + \Phi_t'(\xi_t)( b_{\mathrm{BMD}}(\boldsymbol{s}(t) - \hat{\xi}_t) + b(\boldsymbol{s}(t) - \hat{\xi}_t))dt \\ &+ \frac{1}{2}\Phi_t''(\xi_t)((\alpha(\boldsymbol{s}(t) - \hat{\xi}_t))^2 -6)dt,  
\end{split}\end{equation}
where $b_{\mathrm{BMD}}(\boldsymbol{s}) := 2\pi \lim_{z\to 0}\biggl(\Psi_{\boldsymbol{s}}(z,0) + \frac{1}{\pi z}\biggr) $  is the BMD domain constant which expresses the discrepancy between standard slit domain $\mathbf{H}\setminus \boldsymbol{s}$ and $\mathbf{H}$. As a result, by introducing a reparametrisation $\widetilde{F}_t = F_{c^{-1}(2t)}, \widetilde{W}_t = W_{c^{-1}(2t)}, \widetilde{\xi}_t = \xi_{c^{-1}(2t)}, \widetilde{g}_t^0 = g^0_{c^{-1}(2t)}, \widetilde{\Phi}_t = \Phi_{c^{-1}(2t)}$, where $c_t = 2\int_0^t(\Phi_r'(\xi_r))^2dr $, it follows that $\{\widetilde{F}_t\}_{t\geq 0}$ has the same distribution as a Loewner chain driven by semimartingale $\widetilde{W}_t $. A semimartingale decomposition of $\widetilde{W}_t$ can be easily derived from \eqref{sem-decop} and it is given by 
\begin{equation}\begin{split}\label{sem-decop-2}
d\widetilde{W}_t = \alpha(\widetilde{\boldsymbol{s}}(t) &- \hat{\widetilde{\xi}}_t)d\widetilde{B}_t + \widetilde{\Phi}_t'(\widetilde{\xi}_t)^{-1}( b_{\mathrm{BMD}}(\widetilde{\boldsymbol{s}}(t) - \hat{\widetilde{\xi}}_t) + b(\widetilde{\boldsymbol{s}}(t) - \hat{\widetilde{\xi}}_t))dt \\ &+ \frac{1}{2}\widetilde{\Phi}_t''(\widetilde{\xi}_t)\widetilde{\Phi}_t'(\widetilde{\xi}_t)^{-2}((\alpha(\widetilde{\boldsymbol{s}}(t) - \hat{\widetilde{\xi}}_t))^2 -6)dt.
\end{split}\end{equation}
Note that $\Phi_t$ is univalent on $D_t \cup \overline{D}_t \cup \partial \mathbf{H}$ and $\Phi_t'(z)$ is non-vanishing. Moreover, since $b$ and $b_{\mathrm{BMD}}$ are locally Lipschitz, the drift term is bounded whenever $\boldsymbol{s}(t)$ is bounded away from the real axis. Hence, Theorem~\ref{main-thm-2} immediately implies the following result.

\begin{proposition}
Let $\alpha$, $b$ be locally Lipschitz continuous and homogeneous of degree $0$, $-1$, respectively. Suppose either $\sup_{\boldsymbol{s}}\alpha(\boldsymbol{s})^2 < 8$ or $8 < \inf_{\boldsymbol{s}}\alpha(\boldsymbol{s})^2 \le \sup_{\boldsymbol{s}}\alpha(\boldsymbol{s})^2 < \infty$. Then SKLE$_{\alpha, b}$ is almost surely generated by a curve for $t \le \tau$ for every time $\tau$ where $\dist(F_\tau,\boldsymbol{s}) > 0$. Furthermore, if $\sup_{\boldsymbol{s}}\alpha(\boldsymbol{s})^2 < 4$, then SKLE$_{\alpha,b}$ is a simple curve for $t \le \tau$.
\end{proposition}

\subsection{Other prospective applications}

Applications such as Corollary \ref{imp-cor} was one of our initial motivation behind proving Theorem \ref{main-thm-1} and \ref{main-thm-2}. A related question was asked by A. Sep\'ulveda (in a private communication with the second author): Is Loewner chain driven by $|B_t|$ generated by a curve? His motivation behind posing this problem is to study SLEs with reflecting barriers. Such drivers naturally fall in the class of semimartingales. The main hurdle while adapting the proof of \cite{rs-sle} to such drivers is that in \cite{rs-sle} one has to do some exact computations of certain ``martingale observables". Such exact computations are not feasible (or at least very difficult) for $|B_t|$. Our premise behind proving Theorem \ref{main-thm-1} and \ref{main-thm-2} is to develop techniques that can bypass these exact computations and possibly apply to driving functions such as $|B_t|$. Even though our technique still falls short to handle $|B_t|$ (the local time is not of finite energy), we could handle $|B_t|^\alpha$ for $\alpha >3/2$ as shown in Corollary \ref{imp-cor}. 

A. Sep\'ulveda asked a yet another related question (in a private communication with the second author): what are scaling limits of gluing of two different statistical mechanics model, e.g. percolation model on the upper half plane with the Ising model on the lower half plane stitched together on the boundary in a certain fashion. While this problem is stated imprecisely, it is natural to expect that such scaling limits, if they exists, are SLEs with non-constant $\kappa$. Semimartingales drivers are a natural framework to include such cases.

Another outlook of this paper is to understand the problem of the existence of trace $\gamma$ more deeply by enlarging the class of drivers which do produce a curve. We expect that having more examples of drivers which do produce the trace will provide more insight into the existence problem. This in turn will help us to have a better understanding of another problem: when is the map $W \mapsto \gamma$ continuous? This is also linked to a related problem of \textit{Continuity in $\kappa$ of SLE$_\kappa$}, see \cite{vrw-continuity-kappa,fty-regularity-kappa} for some results in this direction.

\bibliographystyle{alpha}

\end{document}